\documentclass[12pt, a4paper]{amsart}
\usepackage{amssymb,fullpage}
\usepackage{amsmath,esint,enumitem}
\usepackage{colortbl}
\usepackage[dvipsnames]{xcolor}
\usepackage[colorlinks, allcolors=RedViolet,pdfstartview=,pdfpagemode=UseNone]{hyperref} 
\usepackage{graphics}

\newtheorem{theorem}[equation]{Theorem}%[section]
\newtheorem{lemma}[equation]{Lemma}
\newtheorem{proposition}[equation]{Proposition}
\newtheorem{corollary}[equation]{Corollary}

\theoremstyle{definition}
\newtheorem{definition}[equation]{Definition}

\newtheorem{example}[equation]{Example}

\theoremstyle{remark}
\newtheorem{remark}[equation]{Remark}
\numberwithin{equation}{section}

\let\oldmarginpar\marginpar
\renewcommand\marginpar[1]{\-\oldmarginpar[\raggedleft\footnotesize #1]%
{\raggedright\footnotesize #1}}

\newcommand{\supp}{\operatorname{supp}}

\newcommand{ \R }{ \mathbb{R} }
\newcommand{ \Rn }{ {\mathbb{R}^n} }

\newcommand{\Phiw}{{\Phi_{\mathrm{w}}}}
\newcommand{\Phic}{{\Phi_{\mathrm{c}}}}
\newcommand{\bphi}{{\bar\phi}}
\newcommand{\bA}{{\bar A}}
\newcommand{\bF}{{\bar F}}
\newcommand{\bu}{{\bar u}}

\newcommand{\bsigma}{\bar \sigma}
\newcommand{\bomega}{\bar \omega}

\renewcommand{\epsilon}{\varepsilon}
\renewcommand{\phi}{\varphi}
\renewcommand{\le}{\leqslant}
\renewcommand{\ge}{\geqslant}
\renewcommand{\leq}{\leqslant}
\renewcommand{\geq}{\geqslant}
\renewcommand{\div}{\operatorname{div}}

\newcommand{\loc}{\mathrm{loc}}

\newcommand{\ainc}[1]{\hyperref[ainc]{{\normalfont(aInc){\ensuremath{_{#1}}}}}}
\newcommand{\adec}[1]{\hyperref[adec]{{\normalfont(aDec){\ensuremath{_{#1}}}}}}
\newcommand{\inc}[1]{\hyperref[inc]{{\normalfont(Inc){\ensuremath{_{#1}}}}}}
\newcommand{\dec}[1]{\hyperref[dec]{{\normalfont(Dec){\ensuremath{_{#1}}}}}}
\newcommand{\azero}{\hyperref[azero]{{\normalfont(A0)}}}
\newcommand{\aone}{\hyperref[aone]{{\normalfont(A1)}}}

\newcommand{\wVA}{\hyperref[wVA1]{{\normalfont(wVA1)}}}
\newcommand{\VA}{\hyperref[VA1]{{\normalfont(VA1)}}}
\newcommand{\aonen}[1]{\hyperref[aonen]{{\normalfont(A1-\ensuremath{#1})}}}
\newcommand{\VAn}[1]{\hyperref[VA1n]{{\normalfont(VA1-\ensuremath{#1})}}}
\newcommand{\wVAn}[1]{\hyperref[wVA1n]{{\normalfont(wVA1-\ensuremath{#1})}}}
\newcommand{\sVAn}[1]{\hyperref[sVA1n]{{\normalfont(sVA1-\ensuremath{#1})}}}

\begin{document}

\title{Regularity theory for non-autonomous problems with a priori assumptions}

% Information for first author
\author{Peter H\"ast\"o}
 % Address of record for the research reported here

\address{Department of Mathematics and Statistics, FI-20014 University of Turku, Finland}
\email{peter.hasto@utu.fi}

\author{Jihoon Ok}
 % Address of record for the research reported here
\address{Department of Mathematics, Sogang University, Seoul 04107, Republic of Korea}
\email{jihoonok@sogang.ac.kr}

% \thanks will become a 1st page footnote.
\thanks{}

% General info
\date{\today}

\begin{abstract}
We study weak solutions and minimizers $u$ of the non-autonomous problems 
$\operatorname{div} A(x, Du)=0$ and $\min_v \int_\Omega F(x,Dv)\,dx$
with quasi-isotropic $(p, q)$-growth. 
We consider the case that $u$ is bounded, H\"older continuous or lies in a 
Lebesgue space and establish a sharp connection between 
assumptions on $A$ or $F$ and the corresponding norm of $u$. 
We prove a Sobolev--Poincar\'e inequality, higher integrability and the 
H\"older continuity of $u$ and $Du$. Our proofs are optimized and streamlined 
versions of earlier research that can more readily be further extended to other settings. 

Connections between assumptions on $A$ or $F$ and assumptions on $u$ 
are known for the double phase energy
$F(x, \xi)=|\xi|^p + a(x)|\xi|^q$. We 
obtain slightly better results even in this special case. Furthermore, we also cover 
perturbed variable exponent, Orlicz variable exponent,
degenerate double phase, Orlicz double phase,
triple phase, double variable exponent as well as variable exponent double phase energies and 
the results are new in most of these special cases. 
\end{abstract}

\maketitle

%\tableofcontents

%%%%%%%%%%%%%%%%%%%%%%%%%%%%%%%%%%%%%%%%%%%%%%%%%%%%%%%%%%%%%%%%%%%%%%%%%%%%%%%%%%%%%%%%%%%%%%%%
%%%%%%%%%%%%%%%%%%%%%%%%%%%%%%%%%%%%%%%%%%%%%%%%%%%%%%%%%%%%%%%%%%%%%%%%%%%%%%%%%%%%%%%%%%%%%%%%
%%%%%%%%%%%%%%%%%%%%%%%%%%%%%%%%%%%%%%%%%%%%%%%%%%%%%%%%%%%%%%%%%%%%%%%%%%%%%%%%%%%%%%%%%%%%%%%%

\section{Introduction}

We consider the divergence form, quasilinear elliptic equation 
\begin{equation}\label{mainPDE}
\tag{{\(\div A\)}}
\div A(x,Du) = 0 \quad\text{in }\ \Omega,
\end{equation}
and the corresponding $F$-energy minimization
\begin{equation}\label{mainfunctional}
\tag{{\(\min F\)}}
\min_u \int_{\Omega} F(x,Du)\, dx, 
\end{equation}
where $A$ and $F$ have quasi-isotropic $(p,q)$-growth (see Definition~\ref{def:AF}). 
Since we allow $A$ and $F$ to depend on $x$, these are \textit{non-autonomous} problems.
The strategy for dealing with 
non-autonomous problems is often the reduction to and approximation with autonomous problems, 
such as the $p$-power energy $F(\xi):=|\xi|^p$, $p\in (1,\infty)$ and 
the $p$-Laplace equation with $A(\xi):=|\xi|^{p-2}\xi$. 
The maximal regularity of weak solutions already to the $p$-Laplace equation when $p\ne 2$ 
is $C^{1,\alpha}$ for some $\alpha\in(0,1)$  (e.g., \cite{DiBe1,Eva3,Le1, Man86,Ura1}) and this is 
the objective also in more general cases, including in this article.
The approximation technique is often used to deal with 
Marcellini's \cite{Mar89} \textit{$(p,q)$-growth} energies, 
$|\xi|^p \lesssim F(\xi)\lesssim |\xi|^q+1$ and $1<p\le q$, provided that $\frac qp$ is close 
to $1$, see, e.g., \cite{BecMin20,BelSch20,DeFMin21,DeFMin22,Mar91}.

To explain the objective of the current paper we consider the \textit{double phase functional} 
$F(x,\xi):= |\xi|^p + a(x) |\xi|^q$ with $1<p\le q$ and $a:\Omega\to [0, L_\omega]$, which is a special 
case of $(p, q)$-growth. 
This model was first studied by Zhikov \cite{Zhi86, Zhi95} in the 1980's and has 
recently enjoyed a resurgence after a series of papers by Baroni, Colombo and Mingione 
\cite{BarCM15, BarCM16, BarCM18, ColM15a, ColM15b, ColM16}. They studied the relationship between 
the parameters $p$ and $q$ and the H\"older-exponent $\alpha$ of $a$ and 
established maximal regularity of the minimizer $u$ in the following three cases
of a priori information:
\begin{enumerate}
\item[(ap1)]
$u\in W^{1,p}(\Omega)$ and $q-p\le \frac pn \alpha$ 
\item[(ap2)]
$u\in L^\infty(\Omega)$ and $q-p\le \alpha$ 
\item[(ap3)]
$u\in C^{0,\gamma}(\Omega)$ and $q-p< \frac 1{1-\gamma} \alpha$ 
\end{enumerate}
Furthermore, in the first two cases the inequality is sharp in the sense that there exist counter-examples 
to regularity which fail the inequality arbitrarily little \cite{BalDS20,ELM04}. 
The case of equality in (ap3) is an open problem. 
Ok \cite{Ok16} added to these a fourth, likewise sharp, case:
\begin{enumerate}\setcounter{enumi}{3}
\item[(ap4)]
$u\in L^{s^*}(\Omega)$ and $q-p\le \frac sn \alpha$ 
\end{enumerate}
In view of the Sobolev embedding when $s<n$, $s=n$ and $s>n$, this suggests the unifying, 
albeit slightly stronger, assumption 
$u\in W^{1,s}(\Omega)$ and $q-p\le \frac sn \alpha$.

While the relationship between a priori information on $u$ and the conditions 
for double phase $F$ are quite well understood, this is not the case for the 
wide range of recently introduced double phase variants, which extend it or combine it with the 
variable exponent case $F(x,\xi)=|\xi|^{p(x)}$ \cite{DieHHR11, Rad15}.
These variants include 
perturbed variable exponent,
Orlicz variable exponent,
degenerate double phase,
Orlicz double phase,
triple phase,
double variable exponent and
variable exponent double phase. See 
Corollary~\ref{cor:holder-bounded-special} for the corresponding expressions $F$ 
and Table~\ref{table:examples} for examples of our assumptions in some 
of these cases. We refer to \cite{HasO22} for references up to 2020 and \cite{BaaBL22, BaaBO21, CreGHW22, DeF22, FanRZZ22, HadSV_pp, MaeMOS21, MizOS21} 
for some more recent advances on variants of the variable exponent and double phase models.

In most of the special cases, both lower and maximal regularity remain unstudied under 
assumptions (ap2)--(ap4). Recently, Baasandorj and Byun \cite{BaaB_pp1} 
proved maximal regularity of the Orlicz triple phase case in a massive paper.
Rather than study each case individually, 
we introduced an approach based on generalized Orlicz spaces 
in \cite{HasO22} and proved maximal regularity 
for minimizers when $F(x,\xi)=F(x,|\xi|)$ has so-called \textit{Uhlenbeck structure}. 
In \cite{HasO22b} we extended the results to  
weak solutions and minimizers of problems  with $(p,q)$-growth without the Uhlenbeck restriction.
In both articles we only considered the assumption 
$u\in W^{1,\phi}(\Omega)$ corresponding to case (ap1). 
In this article we cover all the different assumptions from cases (ap1)--(ap4), 
including as special cases all the 
double phase variants listed in the previous paragraph. 

We build on the harmonic approximation approach from \cite{BarCM18}. 
Our method is more streamlined and we are 
even able to improve the results in the double phase case slightly by introducing the following 
version of (ap2), which is natural to expect based on the intuition of the Sobolev embedding, 
and a version of (ap3) with equality provided we have vanishing H\"older continuity:
\begin{enumerate}\setcounter{enumi}{1}
\item[(ap2$'$)]
$u\in BMO(\Omega)$ and $q-p\le \alpha$
\item[(ap3$'$)]
$u\in VC^{0,\gamma}(\Omega)$ and $q-p\le \frac 1{1-\gamma} \alpha$ 
\end{enumerate}
This article 
represents a substantial generalization and unification of prior theory. 
We expect that our optimized methods can more readily be further extended to other settings.

In recent years, many papers consider bounded weak solutions or minimizers (i.e.\ Case~(ap2)), see for instance 
\cite{BaaB_pp1, BenK20, ColM15b, HarH21}. The boundedness can be naturally obtained from the maximum principle for bounded Dirichlet boundary value
problems, and is thus a fundamental assumption. 
The following special case of Corollary~\ref{Cor:BMO} showcases our results for $L^\infty$.
We emphasize that even many of these special case results are new and that our main 
results, Theorems~\ref{thm:PDEholder} and \ref{thm:functionalholder} and 
Corollary~\ref{Cor:BMO}, also cover other a priori assumptions and structures.

\begin{corollary}[Bounded minimizers in special cases]\label{cor:holder-bounded-special}
%Assume that $F(x,\xi):=F(x,|\xi|)$ equals one of following functions, where $1<p\le q$,
%$\phi, \psi\in\Phic$ satisfy \ainc{} and \adec{} and $\psi\circ\phi^{-1}$ is almost increasing. 
%We assume that the variable exponents are H\"older continuous and bounded 
%away from $1$ and $\infty$, that $a\in C^{0,\alpha_a}$ and $b\in C^{0,\alpha_b}$ are non-negative, 
%and that the listed additional conditions hold: 
Let $1<p< \min\{q,r\}$,
%$\phi, \psi\in C^{1}([0,\infty))\cap C^{2}((0,\infty))$ such that $\partial_t \phi$ and $\partial_t \psi$ satisfy \inc{p-1} and \adec{q-1} and $\psi\circ\phi^{-1}$ is almost increasing,
 the variable exponents $p(x)$ and $q(x)$ with $p(x)\le q(x)$ be H\"older continuous and bounded 
away from $1$ and $\infty$, and  $a\in C^{0,\alpha_a}$ and $b\in C^{0,\alpha_b}$ be non-negative and bounded. Assume that $F(x,\xi)=f(x,|\xi|)$ equals one of following functions with corresponding additional conditions hold: 
 
\smallskip
\renewcommand{\arraystretch}{1.3}
\centerline{\begin{tabular}{lll}
%\hline
Model & $f(x,t)$ & Additional condition \\
\hline
\emph{Perturbed double phase}&
$t^p + a(x) t^q\log(e+t)$ & $\alpha_a\ge q-p$
%\emph{Orlicz double phase}&
%$\phi(t) + a(x) \psi(t)$ & $a\in C^{\lambda}$ for 
%$\lambda(t):=\frac{\phi(t^{-1})}{\psi(t^{-1})}$ 
 %\hfill(Orlicz double phase)
\\
\emph{Triple phase}&
$t^p + a(x) t^q + b(x) t^r$ & $\alpha_a\ge q-p$ \& $\alpha_b\ge r-p$ %\hfill(triple phase)
\\
\emph{Variable exponent double phase}&
$t^{p(x)} + a(x) t^{q(x)}$ & $\alpha_a(x)\ge q(x)-p(x)$ %\hfill(variable exponent double phase) 
\\
%\hline
\end{tabular}}
\smallskip
\noindent 
Then every minimizer $u\in W^{1,1}_{\loc}(\Omega)\cap L^\infty(\Omega)$ 
of \eqref{mainfunctional} satisfies $u\in C^{1,\alpha}_{\loc}(\Omega)$ for some $\alpha\in(0,1)$ 
independent of $\|u\|_{L^\infty(\Omega)}$.
\end{corollary}

\begin{remark}
The previous corollary holds also with the weaker, but more difficult to check, assumption 
$u\in W^{1,1}_{\loc}(\Omega)\cap BMO(\Omega)$. Without the additional a priori information 
$u\in BMO(\Omega)$, the additional conditions are 
\[
\alpha_a\ge \frac np(q-p),\ \alpha_b\ge \frac np(r-p), 
\text{ and } \alpha_a(x)\ge \frac n{p(x)}(q(x)-p(x)).
\]
These are stronger assumptions when $p<n$ as expected, since if $p\ge n$, then 
$W^{1,p}_\loc(\Omega) \subset BMO_\loc(\Omega)$, so the a priori assumption 
$u\in BMO(\Omega)$ actually contains no additional information.
\end{remark}

\begin{remark}
Our results also apply in the following cases with the same assumptions as in Corollary~\ref{cor:holder-bounded-special}. 

\smallskip
\renewcommand{\arraystretch}{1.3}
\centerline{\begin{tabular}{ll}
%\hline
Model & $f(x,t)$  \\
\hline
{Perturbed variable exponent}&
$t^{p(x)} \log(e+t)$ %$\alpha_p>0$ %\hfill (perturbed variable exponent)
\\
{Orlicz variable exponent}&
$\phi(t)^{p(x)}$\ \ or\ \ $\phi(t^{p(x)})$  %$\alpha_p>0$ %\hfill(Orlicz variable exponent)
\\
%$\phi(t^{p(x)})$ & $\alpha_p>0$ %\hfill(Orlicz variable exponent)
%\\
{Double variable exponent}&
$t^{p(x)}+t^{q(x)}$  %$\alpha_p,\alpha_q>0$ %\hfill(double variable exponent) 
\\
{Degenerate double phase}&
$t^p + a(x) t^p \log (e+t)$  %$\alpha_a>0$ %\hfill(degenerate double phase)
\\
%\hline
\end{tabular}}
\smallskip
\noindent 
However, in these cases the a priori information does not give any improvement 
in the result. The reason is that for these energies, a calculation shows that 
\aonen{n} holds if and only if the \aone{} condition holds, see Section~\ref{sect:lower}. 
In other words, for these cases we obtain a new proof of results previously obtained 
in \cite{HasO22b}.

On the other hand, a priori information does matter for the \emph{Orlicz double phase} 
$\phi(t) + a(x) \psi(t)$, where $a\in C^{0,\lambda}$ and $\psi/\phi$ is almost increasing, 
but the conditions get a bit messy.  
The main condition is that for each $\epsilon>0$ there exists $\beta>0$ such that 
\[
%\lambda(r) \frac{\psi(r^{-(1-\epsilon)})}{\phi(r^{-(1-\epsilon)})} \lesssim r^\beta
\lambda(r) \frac{\psi(r^{-1/(1+\epsilon)})}{\phi(r^{-{1/(1+\epsilon)}})} \lesssim r^\beta.
\]
The detailed calculations are left to the interested reader, 
cf.\ \cite[Corollary~8.4]{HasO22}. 
\end{remark}

We study regularity of weak solutions or minimizers $u$ with the additional 
information that they belong to $L^{s^*}$, $BMO$, $L^\infty$ or $C^{0,\gamma}$, 
and study sharp conditions on $A$ or $F$ corresponding to 
restrictions on $u$. See Definition~\ref{def:continuity} for these sharp conditions and 
Example~\ref{ex:doublephase} for their interpretation in the double phase case. 
The functions 
$A:\Omega\times\R^n\to \R^n$ from \eqref{mainPDE} and $F:\Omega\times \R^n \to \R$ from 
\eqref{mainfunctional} have quasi-isotropic $(p,q)$-growth structure, given in Definition~\ref{def:AF}.
%The quasi-isotropic $(p,q)$-growth structure of
%$A:\Omega\times\R^n\to \R^n$ from \eqref{mainPDE} and $F:\Omega\times \R^n \to \R$ from 
%\eqref{mainfunctional} is given in Definition~\ref{def:AF}.
We briefly explain the strategy and structure of the paper. 

In Section~\ref{sect:lower}, we consider lower order regularity in cases of generalized Orlicz growth
and a priori information. 
We prove $C^{0,\alpha}$-regularity for some $\alpha\in(0,1)$ and higher integrability for quasiminimizers 
(Theorems~\ref{thm:holder} and \ref{thm:reverseHolder}). These are based on Sobolev--Poincar\'e type 
inequalities with a priori information, which are obtained in Theorem~\ref{thm:poincare} 
with Lemma~\ref{lem:poincare_sufficient}. 
%We introduce quasi-isotropic $(p,q)$-growth conditions for $A$ or $F$ and 
%its growth function $\phi$ and a harmonic approximation lemma in Section~\ref{sect:growthFunctions}.

In Section~\ref{sect:regularity} we prove the main results, 
maximal regularity of weak solutions and minimizers for cases (ap2)--(ap4). 
We prove $C^{0,\alpha}$-regularity for every $\alpha\in(0,1)$ and $C^{1,\alpha}$-regularity for some 
$\alpha\in(0,1)$ assuming a priori $C^{0,\gamma}$-information (Theorems~\ref{thm:PDEholder} and \ref{thm:functionalholder}). 
Other cases follow as corollaries by the lower order regularity results. 
The crucial step of the proofs is approximating the original problem 
\eqref{mainPDE} and \eqref{mainfunctional} with a suitable autonomous 
problem and obtaining a comparison estimate between solutions to the original problem and 
the autonomous problem. 
For the approximation we use tools from \cite{HasO22b}, but the comparison 
is achieved quite differently from our earlier papers. 
In this paper, we use harmonic approximation in Lemma~\ref{lem:har1} generalizing 
the double phase case from \cite{BarCM18}. The main innovations are 
inventing assumptions and formulating results optimally to cover all 
special cases while also being sharp, see comments before Theorem~\ref{thm:PDEholder} for details. 

We start in Section~\ref{sect:preliminaries} by recalling 
notation, definitions and basic results on generalized Orlicz spaces.

%%%%%%%%%%%%%%%%%%%%%%%%%%%%%%%%%%%%%%%%%%%%%%%%%%%%%%%%%%%%%%%%%%%%%%%%%%%%%%%%%%
%%%%%%%%%%%%%%%%%%%%%%%%%%%%%%%%%%%%%%%%%%%%%%%%%%%%%%%%%%%%%%%%%%%%%%%%%%%%%%%%%%
%%%%%%%%%%%%%%%%%%%%%%%%%%%%%%%%%%%%%%%%%%%%%%%%%%%%%%%%%%%%%%%%%%%%%%%%%%%%%%%%%%

\section{Preliminaries and notation}\label{sect:preliminaries}
%\subsection*{Notation and generalized Orlicz spaces}

Throughout the paper we always assume that $\Omega$ is a bounded domain in $\Rn$ with $n \ge 2$. For 
$x_0\in \Rn$ and $r>0$, $B_r(x_0)$ is the open ball with center $x_0$ and radius $r$. 
If its center is clear or irrelevant, we write $B_r=B_r(x_0)$. 
The \textit{characteristic function} $\chi_E$ of $E\subset \R^n$ is defined as $\chi_E(x)=1$ if $x\in E$ and $\chi_E(x)=0$ if $x\not\in E$. 

Let $f,g:E \to \R$ be measurable in $E\subset \Rn$. 
We denote the integral average of $f$ over $E$ with $0<|E|<\infty$ by 
$(f)_E:=\fint_E f\, dx := \frac{1}{|E|}\int_E f \,dx$. 
The gradient of $f$ is denoted $Df$.
If $E\subset \R$, then $f$ is said to be \textit{almost increasing} with constant 
$L\ge 1$ if $f(s)\le L f(t)$ whenever $s\le t$. If $L=1$, 
$f$ is \textit{increasing}. Similarly, we define 
an \textit{almost decreasing} or \textit{decreasing} function.
We write $f\lesssim g$, $f\approx g$ and $f\simeq g$ if there exists $C\ge 1$ such that $f(y)\le Cg(y)$, $C^{-1}f(y)\le g(y) \le Cf(y)$ and $f(C^{-1}y)\le g(y) \le f(Cy)$ for all $y\in E$, respectively. 
We use $c$ as a generic constant whose value may change between appearances. 

A \textit{modulus of continuity} $\omega:[0,\infty)\to [0,\infty)$ is concave and increasing 
with $\omega(0)=\lim_{r\to 0^+}\omega(r)=0$.
We define the H\"older seminorm by 
\[
[u]_{\gamma}=[u]_{\gamma,\Omega} := \sup_{x,y\in\Omega, \, x\ne y } \frac{|u(x)-u(y)|}{|x-y|^\gamma}
\qquad\text{and}\qquad
[u]_{\gamma, r}
%=[u]_{\gamma, r,\Omega}
:=\sup_{x\in\Omega} [u]_{\gamma, B_r(x)\cap \Omega}.
\]
%if we want to specify the set $B_r$ we use the notation $[u]_{\gamma, r}$, etc. 
Vanishing H\"older continuity $VC^{0,\gamma}(\Omega)$ means that $u\in C^{0,\gamma}(\Omega)$ and $\lim_{r\to 0^+}[u]_{\gamma,r}=0$. The spaces $C^{0,\gamma(\cdot)}(\Omega)$ and $C^{0,\log}(\Omega)$, 
as well as their vanishing versions, are defined similarly with 
$|x-y|^{\gamma(x)}$ and $\log(e+\frac1{|x-y|})$ instead of $|x-y|^\gamma$ in the denominator.

We refer to 
\cite[Chapter~2]{HarH19} for the following definitions and properties.

\begin{definition} \label{def:ainc}
We define some conditions for $\phi:\Omega\times[0,\infty]\to [0,\infty)$ and $\gamma\in\R$ related to regularity with 
respect to the second variable, which are supposed to hold for all $x\in \Omega$ and 
a constant $L\ge 1$ independent of $x$. 
\vspace{0.2cm}
\begin{itemize}
\item[\normalfont(aInc)$_\gamma$]\label{ainc} 
$t\mapsto \phi(x,t)/t^\gamma$ is almost increasing on $(0,\infty)$ with constant $L$.
\vspace{0.2cm}
\item[\normalfont(Inc)$_\gamma$]\label{inc} 
$t\mapsto \phi(x,t)/t^\gamma$ is increasing on $(0,\infty)$. 
\vspace{0.2cm}
\item[\normalfont(aDec)$_\gamma$]\label{adec} 
$t\mapsto \phi(x,t)/t^\gamma$ is almost decreasing on $(0,\infty)$ with constant $L$.
\vspace{0.2cm}
\item[\normalfont(Dec)$_\gamma$]\label{dec} 
$t\mapsto \phi(x,t)/t^\gamma$ is decreasing on $(0,\infty)$.
\vspace{0.2cm}
\item[\normalfont(A0)] \label{azero} $L^{-1}\leq \phi(x,1)\leq L$.
\end{itemize}
We write \ainc{} or \adec{} if \ainc{\gamma} or \adec{\gamma} holds for some $\gamma > 1$.
\end{definition}

We can rewrite \ainc{p} or \adec{q} with $p,q>0$ and constant $L\geq 1$ as 
\[
\phi(x,\lambda t)\le L\lambda^p\phi(x,t)
\quad\text{and}\quad
 \phi(x,\Lambda t) \le L \Lambda^q \phi(x,t), \quad \text{respectively},
\]
for all $(x,t)\in \Omega\times [0,\infty)$ and $0\le \lambda\le 1 \le \Lambda$. 
From these inequalities one sees that \ainc{} and \adec{} 
are equivalent to the $\nabla_2$- and $\Delta_2$-conditions, respectively. 
The definition of \azero{} above differs slightly from \cite{HarH19} but the two definitions are equivalent
when $\phi$ satisfies \adec{}. 
If $\phi(x,\cdot)\in C^1((0,\infty))$, then for $0<p \le q$,
\[
\phi\text{ satisfies \inc{p} and \dec{q}}
\quad \Longleftrightarrow \quad 
p \le \frac{t\phi'(x,t)}{\phi(x,t)} \le q \text{ for all } t\in(0,\infty).
\]

Suppose $\phi,\psi : [0,\infty)\to [0,\infty)$ are increasing, $\phi$ satisfies \ainc{1} and \adec{}, and $\psi$ 
satisfies \adec{1}. Then there exist a convex $\tilde\phi$ and a concave $\tilde\psi$ such that $\phi\approx\tilde \phi$ and $\psi\approx\tilde \psi$ \cite[Lemma~2.2.1]{HarH19}. 
Therefore, by Jensen's inequality for $\tilde\phi$ and $\tilde \psi$, 
\begin{equation*}%\label{Jensen}
\phi\left(\fint_\Omega |f|\,dx\right) \lesssim \fint_\Omega \phi(|f|)\,dx 
\quad\text{and}\quad 
\fint_\Omega \psi(|f|)\,dx \lesssim \psi\left(\fint_\Omega |f|\,dx\right)
\end{equation*}
for every $f\in L^1(\Omega)$ with
implicit constants depending on $L$ from \ainc{1} and \adec{} or \adec{1} (via the constants from the equivalence relation).

\medskip

We next introduce classes of $\Phi$-functions and generalized Orlicz spaces following \cite{HarH19}. 
% In the sequel we omit the words ``generalized'' and ``weak'' mentioned in parentheses. 
We are mainly interested in convex functions for minimization problems and related PDEs, but the 
class $\Phiw(\Omega)$ is very useful for approximating functionals.

\begin{definition}\label{defPhi}
Let $\phi:\Omega\times[0,\infty]\to [0,\infty)$. 
Assume $x\mapsto \phi(x,|f(x)|)$ is measurable for every measurable function $f$ on $\Omega$, 
$t\mapsto \phi(x,t)$ is increasing for every $x\in\Omega$, 
and $\phi(x,0)=\lim_{t\to0^+}\phi(x,t)=0$ and $\lim_{t\to\infty}\phi(x,t)=\infty$
for every $x\in\Omega$. Then $\phi$ is called a 
\begin{itemize}
\item[(1)] \textit{$\Phi$-function}, denoted $\phi\in\Phiw(\Omega)$, if it satisfies \ainc{1};
\item[(2)] \textit{convex $\Phi$-function}, denoted $\phi\in\Phic(\Omega)$, if $t\mapsto \phi(x,t)$ is left-con\-tin\-u\-ous and convex for every $x\in\Omega$.
%\item[(3)] $\phi$ is an \textit{(generalized) $N$-function} if the map $t\mapsto\phi(x,t)$ is positive when $t>0$, convex, continuous and for every $x\in\Omega$, and $\lim_{t\to0}\frac{\phi(x,t)}{t}=0$ and $\lim_{t\to\infty}\frac{\phi(x,t)}{t}=\infty$.
\end{itemize}
If $\bphi$ is independent of $x$ and $\phi(x,t):= \bphi(t)$ satisfies 
$\phi\in\Phiw(\Omega)$ or $\phi\in\Phic(\Omega)$, we write $\bphi\in\Phiw$ 
or $\bphi\in\Phic$.
%If $\bphi$ is independent of $x$ we write $\bphi\in\Phiw$ 
%or $\bphi\in\Phic$ if $\phi(x,t):= \bphi(t)$ satisfies 
%$\phi\in\Phiw(\Omega)$ or $\phi\in\Phic(\Omega)$.
\end{definition}

Note that $\Phic(\Omega)\subset \Phiw(\Omega)$ since convexity implies \inc{1}. 
For $\phi,\psi\in \Phiw(\Omega)$ the relation $\simeq$ is weaker than $\approx$, but they are equivalent if $\phi$ and $\psi$ satisfy \adec{}. 
%For $\phi\in\Phiw(\Omega)$, 
We write
\[
\phi^+_{B_r}(t):=\sup_{x\in B_r\cap \Omega}\phi(x,t)
\quad \text{and}\quad
\phi^-_{B_r}(t):=\inf_{x\in B_r\cap \Omega}\phi(x,t).
\]
The (left-continuous) inverse function with respect to $t$ is defined by
\[
\phi^{-1}(x,t):= \inf\{\tau\geq 0: \phi(x,\tau)\geq t\}. 
\]
If $\phi$ is strictly increasing and continuous in $t$, then this is just the normal inverse. 
We define the conjugate function of $\phi\in\Phiw(\Omega)$ by
\[
\phi^*(x,t) :=\sup_{s\geq 0} \, (st-\phi(x,s)).
\] 
The definition directly implies \textit{Young's inequality}
\[
ts\leq \phi(x,t)+\phi^*(x, s)\quad \text{for all }\ s,t\ge 0.
\]
If $\phi$ satisfies \ainc{p} or \adec{q} for some $p, q>1$, 
then $\phi^*$ satisfies \adec{p'} or \ainc{q'}, respectively; 
the prime denotes the H\"older conjugate, $p'=\frac{p}{p-1}$.
%, and for any 
%$s,t\ge 0$ and $\epsilon\in(0,1)$,
%\[
%ts 
%%\leq 
%%\phi(x,\epsilon^{\frac{1}{p}}t)+\phi^*(x,\epsilon^{-\frac{1}{p}}s) 
%\lesssim 
%\epsilon \phi(x,t)+\epsilon^{-\frac{1}{p-1}} \phi^*(x,s)
%\qquad\text{or}\qquad
%ts 
%%\leq 
%%\phi(x,\epsilon^{-\frac{1}{q'}}t)+\phi^*(x,\epsilon^{\frac{1}{q'}}s) 
%\lesssim 
%\epsilon^{-(q-1)} \phi(x,t)+\epsilon \phi^*(x,s).
%\]
%We will refer to any of the previous three formulas as \textit{Young's inequality}. 
We also note that $(\phi^*)^*=\phi$ if $\phi\in \Phic(\Omega)$ by \cite[Theorem~2.2.6]{DieHHR11}. 
%\marginpar{$\epsilon$-version of Young's inequality not used}

If $\phi\in\Phic(\Omega)$, then there exists an increasing and right-continuous 
$\phi':\Omega\times[0,\infty)\to[0,\infty)$ such that 
$$
\phi(x,t)=\int_0^t \phi'(x,s)\, ds.
$$
We collect some results about this (right-)derivative $\phi'$. 

\begin{proposition}[Proposition~3.6, \cite{HasO22}]\label{prop0} 
Let $\gamma>0$ and $\phi\in\Phic(\Omega)$.
\begin{itemize}
\item[(1)] 
If $\phi'$ satisfies \inc{\gamma}, \dec{\gamma}, \ainc{\gamma} or \adec{\gamma}, then $\phi$ satisfies \inc{\gamma+1}, \dec{\gamma+1}, \ainc{\gamma+1} or \adec{\gamma+1}, respectively, with the same constant $L\geq 1$. 
\item[(2)] 
If $\phi$ satisfies \adec{\gamma}, then $(2^{\gamma+1}L)^{-1}t \phi'(x,t) \le \phi(x,t)\le t \phi'(x,t)$.
\item[(3)] 
If $\phi'$ satisfies \azero{} and \adec{\gamma} with constant $L\geq 1$, then $\phi$ also satisfies \azero{}, with constant depending on $L$ and $\gamma$.
\item[(4)] 
$\phi^*(x,\phi'(x,t))\le t\phi'(x,t)$. 
\end{itemize}
\end{proposition}

%We will use the following inequality for continuously differentiable $\Phi$-functions. 
%
%\begin{proposition}[Proposition~3.8, \cite{HasO22}]\label{prop000}
%Let $\phi\in \Phic\cap C^1([0,\infty))$ with $\phi'$ satisfying 
%\inc{p-1} and \dec{q-1} for some $1<p\leq q$. 
%Then, for $\kappa\in(0,\infty)$ and $x,y\in \R^N$, 
%\[
%\phi(|x-y|) \lesssim 
%\kappa\left[\phi(|x|)+\phi(|y|)\right]+ \kappa^{-1}\frac{\phi'(|x|+|y|)}{|x|+|y|}|x-y|^2.
%\]
%%\end{enumerate}
%\end{proposition} 

%%%%%%%%%%%%%%%%%%%%%%%%%%%%%%%%%%%%%%%%%%%%%%%%%%%%%%%%%%%%%%%%%%%%%%%%%%%%
%\subsection*{Generalized Orlicz spaces}

%\medskip

Let $L^0(\Omega)$ be the set of the 
measurable functions on $\Omega$. For $\phi\in\Phiw(\Omega)$, the \textit{generalized Orlicz space} (also known as the \textit{Musielak--Orlicz space}) is defined as
\[
L^{\phi}(\Omega):=\big\{f\in L^0(\Omega):\|f\|_{L^\phi(\Omega)}<\infty\big\},
\] 
with the (Luxemburg) norm 
\[
\|f\|_{L^\phi(\Omega)}:=\inf\bigg\{\lambda >0: \varrho_{\phi}\Big(\frac{f}{\lambda}\Big)\leq 1\bigg\},
\quad\text{where}\quad\varrho_{\phi}(f):=\int_\Omega\phi(x,|f|)\,dx.
\]
We denote by $W^{1,\phi}(\Omega)$ the set of functions $f\in W^{1,1}(\Omega)$ 
with $\|f\|_{W^{1,\phi}(\Omega)}:=\|f\|_{L^\phi(\Omega)}+\big\||Df|\big\|_{L^\phi(\Omega)}<\infty$. 
Note that if $\phi$ satisfies \adec{}, then $f\in L^\phi(\Omega)$ if and only if 
$\varrho_\phi(f)<\infty$, and if $\phi$ satisfies \azero{}, \ainc{} and \adec{}, 
then $L^\phi(\Omega)$ and $W^{1,\phi}(\Omega)$ are reflexive Banach spaces. 
We denote by $W^{1,\phi}_0(\Omega)$ the closure of 
$C^\infty_0(\Omega)$ in $W^{1,\phi}(\Omega)$. For more information about generalized Orlicz 
and Orlicz--Sobolev spaces, we refer to the monographs 
\cite{CheGSW21, HarH19} and also \cite[Chapter~2]{DieHHR11}.

%%%%%%%%%%%%%%%%%%%%%%%%%%%%%%%%%%%%%%%%%%%%%%%%%%%%%%%%%%%%%%%
%%%%%%%%%%%%%%%%%%%%%%%%%%%%%%%%%%%%%%%%%%%%%%%%%%%%%%%%%%%%%%%
%%%%%%%%%%%%%%%%%%%%%%%%%%%%%%%%%%%%%%%%%%%%%%%%%%%%%%%%%%%%%%%

\section{Lower regularity with a priori assumptions}\label{sect:lower}

\subsection*{Continuity assumptions}\label{subsect:newcondition}

The condition \aone{}, introduced in \cite{Has15} (see also \cite{MaeMOS13a}), is a 
``almost continuity'' assumption, which allows the function to jump, but not too much. 
It implies the H\"older continuity of solutions and (quasi)minimizers 
\cite{BenHHK21, HarHL21, HarHT17}. For higher regularity, 
we introduced in \cite{HasO22} a ``vanishing \aone{}'' condition, denoted \VA{}, and 
a weak vanishing version, \wVA{} and generalized them to the quasi-isotropic 
situation in \cite{HasO22b}. The anisotropic condition was further studied in 
\cite{BorC_pp, Has_pp}. These previous studies applied to 
the ``natural'' energy assumption $u\in W^{1,\phi}(\Omega)$, called Case (ap1) 
in the introduction. The \aonen{n} and \aonen{\psi} conditions for a priori energy assumptions 
were developed in \cite{HarHL21} and \cite{BenHHK21} for functions in $L^\infty$ and 
$W^{1,\psi}$, respectively. 
Here we generalize and unify all the conditions for a priori information; 
the most important one for this article is \VAn{s}. 

\begin{definition}\label{def:continuity}
Let $M,N\in\mathbb N$, $G:\Omega\times \R^M \to \R^N$, $\psi:\Omega\times \R^M\to [0,\infty)$, 
$L_\omega >0 $, $r\in (0,1]$ and $\omega:[0,1]\to [0,L_\omega]$. We consider the claim 
\[
|G(x,\xi)-G(y,\xi)|\leq \omega(r)\big(|G(y,\xi)|+1\big) 
\quad\text{when }\ \psi(y,\xi)\in [0,|B_r|^{-1}]
\]
for all $x,y\in B_r\cap \Omega$ and $\xi\in\R^M$.
We say that $G$ satisfies:
\begin{itemize}[leftmargin=5em]
\item[\normalfont(A1-$\psi$)]\label{aone}\label{aonen}
if there exists $L_\omega$ such that the claim holds with 
$\omega\equiv L_\omega$.
\item[\normalfont(VA1-$\psi$)]\label{VA1}\label{VA1n}
if there exists $L_\omega$ and a modulus of continuity 
$\omega$ such that the claim holds.
\item[\normalfont(wVA1-$\psi$)]\label{wVA1}\label{wVA1n}
if it satisfies \VAn{\psi^{1+\epsilon}} for every $\epsilon>0$, with possibly 
different functions $\omega_\epsilon$ but a common $L_\omega$ independent of $\epsilon$.
\end{itemize}
When $\psi(x,\xi)=|\xi|^s$ with $s>0$ we use the abbreviations \aonen{s}, \VAn{s} and \wVAn{s} and in the 
case $\psi=|G|$ we write \aone{}, \VA{} and \wVA{}. 
We also use the definition for $\psi:\Omega\times [0,\infty)\to[0,\infty)$ with the 
understanding that $\psi(x,\xi)=\psi(x,|\xi|)$.
\end{definition}

\begin{table}
\caption{Examples of sufficient conditions in special cases.}\label{table:examples}
\begin{tabular}{lcccc}
%\hline
\\[-10pt]
Model and function & \aonen{s} & \wVAn{s} & \VAn{s} \\[0pt]
\hline\\[-10pt]
Variable exponent\quad $t^{p(x)}$ & $p\in C^{\log}$& $p\in VC^{\log}$& $p\in VC^{\log}$\\[6pt]
Orlicz variable exponent\quad $\psi(t)^{p(x)}$ & $p\in C^{\log}$& $p\in VC^{\log}$& $p\in VC^{\log}$\\[6pt]
Double phase\quad $t^p + a(x)t^q$, $\alpha:=\frac ns (q-p)$
 & $a\in C^\alpha$ & $a\in C^\alpha$ & $a\in VC^\alpha$ \\[6pt]
%Orlicz double phase \quad $\phi(t) + a(x)\psi(t)$ \\ $\lambda(t):=\phi(t^{-n/s})/\psi(t^{-n/s})$ 
% & $a\in C^{\lambda}$ & ${\color{red}a\in C^{\lambda}}$ & $a\in VC^{\lambda}$ \\[6pt]
Double phase variable exponent \\
$t^{p(x)} + a(x)t^{q(x)}$, $p,q\in C^{\log}$, $\alpha:=\frac ns (q-p)$ 
& $a\in C^{\alpha(\cdot)}$ & $a\in C^{\alpha(\cdot)}$ & $a\in VC^{\alpha(\cdot)}$ \\[6pt]
%& ${\sigma_0}(t)t^{s\alpha/n}\lesssim \psi(t)$& $\lim_{t\to\infty}\frac{{\sigma_0}(t)t^{s\alpha/n}}{\psi(t)}=0$ \\
%\hline
\end{tabular}
\end{table}

It can be easily seen that \VAn{s}$\Longrightarrow$\wVAn{s}$\Longrightarrow$\aonen{s}
and 
\VAn{s'}$\Longrightarrow$\VAn{s} if $s'\le s$, similarly for \wVAn{s} and \aonen{s}. 
%On the other hand, \wVAn{s}$\Longrightarrow$\VAn{s'} for any $s'>s$. 
In the case $M=N=1$, these conditions are somewhat differently formulated than in 
earlier papers, but we showed in \cite{HasO22b} that the formulations are
are equivalent to previous versions under natural assumptions on $G$. 

The next example shows the relevance of the parameter $s$ in \aonen{s} in the double phase case. 
Similar relations for other cases are summarized in Table~\ref{table:examples}.

%\newpage
\begin{example}\label{ex:doublephase}
Consider two double phase energies for $1<p \le q$ and $a:\Omega\to [0,L]$. 
\begin{itemize}
\item
Let $\phi_1(x, t) := t^p + a(x)t^q$, $a\in VC^{0,\alpha}(\Omega)$ for some $\alpha\in(0,1]$. \\
If $q-p \le  \tfrac sn \alpha$, then $\phi_1$ satisfies \VAn{s}  with $\omega$ proportional to the modulus of continuity of $a$. 
%\item 
%$\phi_2(x, t) := t^p + a(x)^\alpha t^q$,
%where $\alpha > 0$ and $a\in C^{0,1}(\Omega)$. \\
% If  $q-p\le \tfrac sn \alpha$, then $\phi_2$ satisfies \aonen{s} and \wVAn{s} with $\omega_\epsilon(r) = c r^{\alpha-\frac{n(q-p)}{s(1+\epsilon)}}$;\\
% If $q-p< \tfrac sn \alpha$, then   $\phi_2$ satisfies \VAn{s} with $\omega(r) = cr^{\alpha-\frac{n(q-p)}{s}}$. \\
\item 
Let $\phi_2(x, t) := t^p + a(x)^\alpha t^q$,
where $\alpha>0$ and $a\in C^{0,1}(\Omega)$. \\
 If  $q-p\le \tfrac sn \alpha$, then $\phi_2$ satisfies \aonen{s} and \wVAn{s} with $\omega_\epsilon(r) = c r^{\{\alpha-\frac{n(q-p)}{(1+\epsilon)s}\}\min\{1,\frac1\alpha\}}$;\\
 If $q-p< \tfrac sn \alpha$, then   $\phi_2$ satisfies \VAn{s} with $\omega(r) = c r^{(\alpha-\frac{n(q-p)}{s})\min\{1,\frac1\alpha\}}$. 
\end{itemize}
Note that these conditions can hold for $\frac qp$ arbitrarily large.
We show the second case only since the first case can be obtained in the same way as the second case with $\alpha<1$. We will use the elementary inequality 
\[
|a^\alpha-b^\alpha| \le 
\begin{cases}
|a-b|^\alpha, &  0< \alpha\le 1,\\
c_\alpha\delta^{1-\alpha} |a-b|^\alpha +\delta b^\alpha, &\quad \alpha>1,
\end{cases}
\]
which holds for any $a, b \ge 0$ and $\delta\in (0,1]$; here $c_\alpha>0$ is a constant depending on $\alpha$. The second inequality ($\alpha>1$) follows from Young's inequality 
applied to the right-hand side of $b^\alpha-a^\alpha \le \alpha b^{\alpha-1}(b-a)$ when $b\ge a \ge 0$. 

Suppose that $q-p\le \frac{s}{n}\alpha$ and let $x,y\in B_r$ with $r\in(0,1]$ and $t \in (0,|B_r|^{-\frac{1}{s(1+\epsilon)}}]$ with $\epsilon\ge 0$. Applying the preceding inequality with $a=a(x)$ and $b=a(y)$, we obtain that 
\[
|\phi_2(x,t)-\phi_2(y,t)|  \lesssim  r^\alpha t^q = r^\alpha t^{q-p} t^p  \lesssim r^{\alpha-\frac{n(q-p)}{(1+\epsilon)s}} \phi_2(y,t)
\]
when $0<\alpha\le 1$; in the case $\alpha > 1$, we choose $\delta:=r^{\{\alpha-\frac{n(q-p)}{(1+\epsilon)s}\}\frac1\alpha}$ and find that 
\[\begin{split}
|\phi_2(x,t)-\phi_2(y,t)| & \lesssim (\delta^{1-\alpha} r^\alpha   +\delta a(y)^\alpha) t^q\\
& \lesssim \delta^{1-\alpha} r^{\alpha-\frac{n(q-p)}{(1+\epsilon)s}} t^p  +\delta a(y)^\alpha t^q   = r^{\{\alpha-\frac{n(q-p)}{(1+\epsilon)s}\}\frac1\alpha}\phi_2(y,t).
\end{split}\]
These inequalities imply the desired \aonen{s},  \wVAn{s} and \VAn{s}-conditions.  
%For detailed computations, we refer to the examples of double phase type in \cite{HasO22,HasO22b}.
\end{example}

Let us show how to use the smallness of $\omega$ to obtain the inequality 
from \VAn{s} for a slightly larger range. Intuitively, we shift some power from 
the coefficient to the range. In this proof it is important that the range of $\xi$ in 
the condition is independent of $x$, so the result does not generalize to \VAn{\psi} easily, 
unless $\psi(x,t)=\psi(t)$. 

\begin{proposition}\label{prop:K}
Let $G:\Omega\times \R^M \to \R^N$ satisfy \VAn{s}, $\theta\in [0,1]$ and $r\in (0,1]$. 
If $\omega(r)^{1-\theta}\le \tfrac13$, then for every $x,y \in B_r\cap\Omega$, 
\[
|G(x,\xi)-G(y,\xi)|\leq \omega(r)^\theta \big(|G(y,\xi)|+1\big) 
\quad\text{when }\ 3^n\omega(r)^{n(1-\theta)} |\xi|^s\in [0,|B_r|^{-1}].
\]
\end{proposition}
\begin{proof}
Note by the concavity of $\log$ that $t\log 2\le \log(1+t)\le t$ when $t\in [0,1]$, and 
set 
\[%\begin{equation}\label{kdef}
k:=\Big\lfloor \frac{\log(1+\omega(r)^\theta)}{\log(1+\omega(r))}\Big\rfloor 
\ge \big\lfloor  \omega(r)^{\theta-1} \log 2 \big\rfloor \ge \frac{\log 2}2 \omega(r)^{\theta-1} > \frac{1}3 \omega(r)^{\theta-1} \ge 1,
\]%\end{equation}
where $\lfloor\tau\rfloor$ is the largest integer less than or equal to $\tau\in \R$.
Suppose $x,y\in B_r\cap\Omega$ and $3^n\omega(r)^{n(1-\theta)} |\xi|^s \le |B_r|^{-1}$.
Then $|\xi|^s\le k^n |B_r|^{-1}$ since 
$1\le (3k)^n\omega(r)^{n(1-\theta)}$. 
We split the segment $[x, y]$ into $k$ equally long subsegments $[x_i, x_{i+1}]$ with $x_0=x$ and $x_k=y$
so that $x_i, x_{i+1}\in B_{r/k}$. Since $|\xi|^s\le |B_{r/k}|^{-1}$, we can use \VAn{s}
to estimate
\[
\begin{split}
|G(x_i, \xi)|+1
&\le 
(1+\omega(r))(|G(x_{i+1}, \xi)| + 1)
%\le 
%(1+\omega(r))^2(|G(x_{i+2}, \xi)| + 1)
%\\
\le \cdots \le 
(1+\omega(r))^{k-i}(|G(y, \xi)| + 1).
\end{split}
\]
We use this estimate with the triangle inequality and 
%\VAn{\frac n{1-\gamma}}:
\VAn{s}:
\[
\begin{split}
|G(x,\xi)-G(y,\xi)| 
&\le
\sum_{i=0}^{k-1} |G(x_{i+1},\xi)-G(x_i,\xi)| 
\le
\omega(r)\sum_{i=1}^{k} (|G(x_i,\xi)| +1) \\
&\le 
\omega(r)  \sum_{i=1}^{k} (1+\omega(r))^{k-i} (|G(y,\xi)| +1)\\
&=
[(1+\omega(r))^k -1](|G(y,\xi)| +1).
\end{split}
\]
This gives the desired estimate, since by the definition of $k$, 
\[
(1+\omega(r))^k
\le
(1+\omega(r))^{\frac{\log(1+\omega(r)^\theta)}{\log(1+\omega(r))}}
=
1+\omega(r)^\theta. \qedhere
\]
\end{proof}

%%%%%%%%%%%%%%%%%%%%%%%%%%%%%%%%%%%%%%%%%%%%%%%%%%%%%%%%%%%%%%%%%%%%%%%

\subsection*{Sobolev--Poincar\'e inequality} 
We derive a modular Sobolev--Poincar\'e-type inequality in generalized Orlicz spaces assuming 
a priori information. 
We first state the inequality with an abstract condition, which is explored further in 
Lemma~\ref{lem:poincare_sufficient} and Example~\ref{eg:poincare_necessary}. The 
example shows that the conditions in 
Lemma~\ref{lem:poincare_sufficient} are essentially sharp for the Sobolev--Poincar\'e inequality, 
at least when $s\le n$. This approach is inspired by \cite{BenHHK21}.

\begin{theorem}[Sobolev--Poincar\'e inequality]\label{thm:poincare} 
Let $\phi\in \Phiw(B_r)$ satisfy \azero{}, \ainc{p} and \adec{q} with $1\le p\le q$, 
and let $u\in W^{1,1}(B_r)$. If
\begin{equation}\label{eq:poin_ass}
\fint_{B_r} \Big(\frac{\phi(x,v)}{\phi_{B_r}^-(v)+1}\Big)^\theta \,dx \le b_0 \quad \text{for }\ v:=\frac{|u-(u)_{B_r}|}r
%\quad\text{ when } p\le n
%\qquad\text{or}\qquad 
%\frac{\phi(x,v)}{\phi_{B_r}^-(v)}\le b_0\quad \text{when } p>n,
\end{equation}
and some $b_0,\theta>0$, then 
\begin{equation*}%\label{sobo-poin}
\bigg(\fint_{B_r} \phi(x,v)^{\theta_0}\,dx\bigg)^{\frac1{\theta_0}}
\le 
c \fint_{B_r} \phi^{-}_{B_r}(|D u|)\,dx + c 
%\le c \fint_{B_r} \phi(x,|D u|)\,dx + c
\end{equation*}
for $\frac1{\theta_0} = 1-\min\{\frac pn,\kappa\}+\frac1\theta$ with any $\kappa\in (0,1)$ 
and some $c=c(n,p,q,L,\kappa,b_0)>0$. 
\end{theorem}

\begin{remark}\label{rem:sobopoinTheta}
In the previous theorem we can choose $\theta_0>1$ if and only if $\theta>\max\{1,\frac np\}$.
The choice $\theta_0=1$ is additionally possible when $\theta=\frac np > 1$. These are the most important cases, but the theorem allows also for $\theta_0\in (0,1)$ in cases with small $\theta$.
\end{remark}

\begin{remark}\label{rem:sobopoin}
Let $1<p \le q$. Suppose that $\phi$ satisfies \eqref{eq:poin_ass} with $\theta$ such that $\theta_0>1$. 
Then so does $\phi^{1/\tau}$ with $\theta':=\theta\tau$ and $\tau \in (1,p)$. 
Theorem~\ref{thm:poincare} for $\phi^{1/\tau}$ implies that 
\[
\bigg(\fint_{B_r} \phi(x,v)^\frac{\theta_0'}{\tau}\,dx\bigg)^{\frac{\tau}{\theta_0'}}
\le c
\left( \fint_{B_r} \phi(x,|D u|)^{\frac{1}{\tau}} \,dx \right)^{\tau}+ c,
\]
where $\frac1{\theta_0'} = 1-\min\{\frac pn,\kappa\}+\frac1{\theta\tau}$. 
Note that $\theta_0'\to\theta_0>1$ when $\tau\to 1$ so we can choose $\tau$ with 
$\frac{\theta_0'}\tau>1$. Thus there exists $\theta_1>1$ depending 
$n$, $p$ and $\theta$ such that
\[
\bigg(\fint_{B_r} \phi(x,v)^{\theta_1}\,dx\bigg)^{\frac{1}{\theta_1}}
\le c
\left( \fint_{B_r} \phi(x,|D u|)^{\frac{1}{\theta_1}} \,dx \right)^{\theta_1}+ c.
\] 
\end{remark}

%\comment{Jihoon: I think the argument in the above remark is quit complicated. How about the following remark instead?
%\begin{remark}
%Suppose ${\color{red}1<p}\le q$, $\theta>\max\{1,\frac np\}$ and $\phi\in\Phiw(B_r)$ the assumption in Theorem~\ref{thm:poincare}. Then there exists $p_1\in(1,p)$ such that $\theta>\max\{1,\frac np_1\}$. Let $\theta_0>1$ be from Theorem~\ref{thm:poincare} with $p$ replaced by $p_1$. We may assume that $\theta_0\le \frac{p}{p_1}$. Then $\phi(x,t)^{\frac{2}{1+\theta_0}}$ satisfies the assumption of $\phi$ in Theorem~\ref{thm:poincare} with $p$ replaced by $p_1$. Therefore, we have 
%\[
%\bigg(\fint_{B_r} \phi(x,v)^{\theta_1}\,dx\bigg)^{\frac{1}{\theta_1}}
%\le c
%\left( \fint_{B_r} \phi(x,|D u|)^{\frac{1}{\theta_2}} \,dx \right)^{\theta_2}+ c,
%\] 
%where $\theta_1=\frac{2\theta_0}{1+\theta_0}>1$ and $\theta_2=\frac{1+\theta_0}{2}>1$.
%\end{remark}
%} 

\begin{proof}[Proof of Theorem~\ref{thm:poincare}]
To obtain a differentiable function, we define $\psi\in \Phic$ by
\[
\psi(t) := 
\int_0^t \sup_{\sigma \in (0,\tau]} \frac{\phi^-_{B_r}(\sigma)}\sigma \, d\tau.
\]
From \ainc{1} of $\phi$ we see that 
$\frac{\phi^-_{B_r}(t)} {t} \le \psi'(t) \le L \frac{\phi^-_{B_r}(t)} {t}$;
with \ainc{1} and \adec{q} we conclude that $\phi_{B_r}^-\approx \psi(t)$. 
Choose $s:=\min\{p,\kappa n\}\in [1,n)$ and $\theta_0\in (0, \theta)$ with 
$\frac sn = 1-\frac1{\theta_0}+\frac1\theta$. 
Note that $\frac{s^*}{\theta_0s}=((1-\frac sn)\theta_0)^{-1}=(1-\frac{\theta_0}\theta)^{-1}>1$. 
By H\"older's inequality with exponents $\frac{s^*}{\theta_0s}$ 
and $(\frac{s^*}{\theta_0s})'=\frac\theta{\theta_0}$
and the assumption \eqref{eq:poin_ass},
\[
\bigg(\fint_{B_r} \phi(x,v)^{\theta_0} \,dx\bigg)^{\frac1{\theta_0}}
\lesssim 
b_0^\frac1\theta \bigg(\fint_{B_r} \psi(v)^{\frac{s^*}s}\,dx +1 \bigg)^{\frac s{s^*}}.
\]

We use the Sobolev--Poincar\'e inequality in $L^s$ and 
$\psi'(t)\approx \psi(t)/t$ to conclude that 
\begin{align*}
\bigg(\fint_{B_r} \big|\psi(v)^{\frac1s}-(\psi(v)^{\frac1s})_{B_r} \big|^{s^*}\,dx\bigg)^{\frac s{s^*}}
& \lesssim
r^s \fint_{B_r} \big|\nabla \big(\psi(v)^{\frac1s}\big)\big|^s \,dx \\
&\approx
r^s \fint_{B_r} \psi(v)^{1-s} \psi'(v)^s |\nabla v|^s \,dx
\approx 
\fint_{B_r} \frac{\psi(v)}{v^s} |\nabla u|^s \,dx.
\end{align*}
Since $\psi$ satisfies \ainc{p} and $s\le p$, $\psi_s(t):=\psi(t^{1/s})$ satisfies \ainc{1}. 
Therefore Young's inequality for $\psi_s$ and $\psi_s^*(\frac {\psi_s(t)}{Lt})\le \psi_s(t)$ 
\cite[Lemma~3.1]{HarHJ_pp} with $t:=v^s$ give 
$\frac{\psi(v)}{v^s} |\nabla u|^s
\lesssim
\psi(v) + \psi(|\nabla u|)$. 
Continuing the previous estimate with the $L^{s^*}$-triangle inequality and $\psi\approx \phi^-_{B_r}$, we find that
\begin{align*}
\bigg(\fint_{B_r} \psi(v)^{\frac{s^*}s}\,dx\bigg)^{\frac s{s^*}}
&\lesssim 
\fint_{B_r} \phi^-_{B_r}(|\nabla u|)\, dx + \fint_{B_r} \psi(v)\,dx
+ (\psi(v)^{\frac1s})_{B_r}^s.
\end{align*}
By H\"older's inequality, $(\psi(v)^{\frac1s})_{B_r}^s\le (\psi(v))_{B_r}$ 
and by the modular Poincar\'e inequality in the Orlicz space $L^\psi$ \cite[Corollary~7.4.1]{HarH19}, 
$(\psi(v))_{B_r}$ can be estimated by the first term on the right-hand side. 
Combined with the inequality from the previous paragraph, this gives the claim. 
\end{proof}

Let us derive some sufficient conditions for the assumption of the previous theorem by 
complementing \cite[Proposition~4.2]{BenHHK21}. Also the cases 
$u\in L^{\psi^\#}(B_r)$ and $u\in W^{1,\psi}(B_r)$ for $\psi\in\Phiw(B_r)$ and \aonen{\psi} 
are covered by \cite{BenHHK21}, and could likewise be considered here. 
We define the \textit{bounded mean oscillation} semi-norm as
\[
[u]_{BMO(B_r)}:= \sup_{B_\rho\subset B_r} \fint_{B_\rho} |u-(u)_{B_\rho}| \,dx.
\]
Note that in case (3) of the following lemma we need 
$s>n(1-\frac pq)$ in order that $\theta_0>1$ in Theorem~\ref{thm:poincare}, 
cf.\ Remark~\ref{rem:sobopoinTheta}. 

\begin{lemma}\label{lem:poincare_sufficient}
Let $\phi\in \Phiw(B_r)$ satisfy \azero{}, \ainc{p}, \adec{q} and \aonen{s} with $1\le p \le q$ 
and let $u\in L^1(B_r)$.
Assume one of the following holds:
\begin{enumerate}
\item
$s>n$ and $[u]_{\gamma, B_r} \le b$ with $\gamma := 1-\frac{n}{s}$. 
\item
$s=n$ and $[u]_{BMO(B_r)}\le b$.
\item
$s\in [1, n)$ and $\|u\|_{L^{s^*}(B_r)} \le b$. 
%\item
%$s \ge 1$ and $\|u\|_{W^{1,s}(B_r)} \le b$.
\end{enumerate}
Then \eqref{eq:poin_ass} holds for any $\theta>0$ in Cases (1)--(2) and for 
$\theta=\frac{s^*}{q-p}$ in Case (3). The constant 
$b_0$ depends only on $n$, $p$, $q$, $L$, $L_\omega$ and $b$.
\end{lemma}
\begin{proof} 
If $[u]_{\gamma}\le b$, then $v\le 2|B_1|^{\frac{1}{n}} b |B_r|^{\frac{\gamma-1}{n}}$%$v\lesssim br^{\gamma-1}$ 
and $\phi^+_{B_r}(v)\lesssim \phi^-_{B_r}(v)+1$ 
by \aonen{s}, so the claim holds in Case (1). For the same reason and \azero{}, the integrand 
in \eqref{eq:poin_ass} in Case (2) 
is bounded at points with $v\le \max\{|B_r|^{-1/n},1\}$. On the other hand, at points with $v>\max\{|B_r|^{-1/n},1\}$
we estimate, by \azero{}, \ainc{p}, \adec{q} and \aonen{n},
\[
\frac{\phi(x,v)}{\phi_{B_r}^-(v)+1}
\approx
\frac{\phi(x,v)}{\phi_{B_r}^-(v)}
\lesssim
\Big(\frac{v}{|B_r|^{-1/n}}\Big)^{q-p} \frac{\phi(x,|B_r|^{-1/n})}{\phi_{B_r}^-(|B_r|^{-1/n})}
\approx
\Big(\frac{v}{r^{-1}}\Big)^{q-p} 
=
|u-(u)_{B_r}|^{q-p} .
\]
We obtain for any exponent $\theta>0$ that 
\[
\fint_{B_r} \Big(\frac{\phi(x,v)}{\phi_{B_r}^-(v)+1}\Big)^\theta\,dx 
\lesssim
\fint_{B_r} |u-(u)_{B_r}|^{\theta(q-p)}\,dx +1
\lesssim
[u]_{BMO(B_r)}^{\theta(q-p)}+1,
\]
where in the last inequality we use the well-known reverse H\"older type inequality for 
mean oscillations in $L^{\theta(q-p)}$-space and $BMO$ (cf. \cite[Lemma A.1]{DieKS12}).
Case (3) was proved in \cite[Proposition~4.2]{BenHHK21}. 
%We note that case (4) is proved in \cite{BenHHK21} using the non-scale invariant Sobolev 
%inequality $\|u\|_{L^{s^*}(B_r)} \lesssim \|u\|_{W^{1,s}(B_r)}$. However, we can instead
%use $\|u-(u)_{B_r}\|_{L^{s^*}(B_r)} \lesssim \|\nabla u\|_{L^s(B_r)}$ and thus the constant is 
%independent of $r$. When $s=n$, we have from Poincar\'e's inequality in $W^{1,1}$ and H\"older's indquality that for every ball $B_\rho(y)$ in $B_r$
%\[
%\fint_{B_\rho(y)} |u-(u)_{B_\rho(y)}|\, dx \le c r \fint_{B_\rho(y)} |Du|\,dx \le c \|Du\|_{L^n(B_\rho(y))} \le c\|Du\|_{L^n(B_r)},
%\]
%hence $[u]_{BMO(B_r)}\le c\|Du\|_{L^n(B_r)}$.
%\comment{Jihoon: When $s>n$, I knwo the the following two estimates $[u]_{1-\frac{n}{s},B_r} \le c(r) \|u\|_{W^{1,s}(B_{r})}$ and $[u]_{1-\frac{n}{s},B_r} \le c \|Du\|_{L^{s}(B_{2r})}$ (I don't think we can replace $B_{2r}$ by $B_r$.) I think we should use the second estimate. My suggestion is that in (4) we assume $u\in W^{1,s}(B_{2r})$ with $\|Du\|_{L^s}\le b$.}
\end{proof}

The estimates for the Sobolev--Poincar\'e inequality may seem crude, but the following example shows that the end result is sharp, i.e.\ the claim is false if \aonen{s} is replaced by \aonen{s'} for any $s'<s$. 
See also \cite[Section~5]{BenHHK21} for a one-dimensional example. 

\begin{example}\label{eg:poincare_necessary}
Let $n=2$ and denote the quadrants by $Q_k\subset\R^2$, $k\in \{1,2,3,4\}$. Let $\eta:[0,4]\to [0,1]$ be the piecewise linear, $3$-Lipschitz function with $\eta_{[\frac13, \frac23]}=\eta_{[\frac73, \frac83]}=1$ 
and $\eta_{[1,2]}=\eta_{[3,4]}=0$. We define $a:\R^2\to [0,\infty)$ in polar coordinates as
$a(r,\theta):=r^\alpha \eta(\frac2\pi \theta)$. 
Thus $a$ equals $0$ in $Q_2$ and $Q_4$ and $a(x)=|x|^\alpha$ in the sectors with $\frac{\pi}{6}<\theta<\frac{\pi}{3}$ in 
$Q_1$ and with $\frac{7\pi}{6}<\theta<\frac{4\pi}{3}$ in $Q_3$. Consider the 
double phase functional $H(x,t)=t^p + a(x)t^q$ with $p<2$ and the
function $u:\R^2\to\R$ which equals $1$ in $Q_1$, $-1$ in $Q_3$ and is linear in 
the polar coordinate $\theta$ in $Q_2$ and $Q_4$. By symmetry, $u_{B_r}=0$ for every ball $B_r$ centered at the 
origin and $v:=\frac1r |u-(u)_{B_r}| = \frac1r$ in the sectors in $Q_1$ and $Q_3$. The derivative of $u$ 
equals zero in $Q_1$ and $Q_3$; in the other quadrants the radial derivative is 
zero, and in the tangential derivative equals $\frac4{\pi r}$. For a constant $k>0$ we estimate, 
based on the sectors in $Q_1$ and $Q_3$, 
\[
\fint_{B_r} H(x, kv)\, dx 
\ge 
\frac{1}{3r^2}\int_0^r s^{1+\alpha} (\tfrac kr)^{q}\, ds = \tfrac1{3(2+\alpha)} r^{\alpha-q} k^q
\]
and, since the support of the derivative is $Q_2 \cup Q_4$ where $a=0$,
\[
\fint_{B_r} H(x, k|\nabla u|)\, dx = \frac{1}{r^2} \int_0^r s(\tfrac {4k}{\pi s})^p\, ds = \tfrac1{2-p} (\tfrac 4\pi)^p r^{-p}k^p. 
\]
If the modular Poincar\'e inequality from Theorem~\ref{thm:poincare} holds with $\theta_0=1$ (the weakest 
relevant case), then 
\[
r^{\alpha-q} k^q \le c r^{-p}k^p + c
\quad\text{so that} \quad 
r^{\alpha-(q-p)} k^{q-p} \le c + cr^{p}k^{-p}.
\]
Suppose that we want the constant in the inequality to depend on the $L^{s^*}$-norm of $ku$. 
We calculate $\|ku\|_{L^{s^*}(B_r)}=ckr^{2/s^*}$. Thus $k = c r^{-2/s^*}$ and so
\[
r^{\alpha-(q-p)} k^{q-p} \approx r^{\alpha - (\frac2{s^*}+1)(q-p)} = r^{\alpha - \frac2s(q-p)}.
\]
This remains bounded as $r\to 0$ when $\alpha \ge \frac2s(q-p)$ which is exactly the \aonen{s} 
condition when $n=2$ and shows the sharpness of Case (3). 
If $s=n=2$, this shows the sharpness of Case (2),
even if we allow the constant to depend on the $L^\infty$-norm. 
Similarly, we see that if the constant is allowed to depend on $\rho_H(|\nabla u|)$, 
then $\alpha \ge \frac2p(q-p)$ which is \aone{}. 
Unfortunately, Case (1) is not covered, since the counter-example is discontinuous. 
%
%The function in this example is discontinuous, so it cannot be used for case (1). Now we can take $k$ not as a 
%constant but as a function $k(r,\theta)=k_0 r^\gamma$. The derivative $\nabla (ku)$ is still dominated by the 
%$\frac1r$-term, and we obtain the condition 
%\[
%c r^{\alpha-q} (r^\gamma k_0)^q \le c r^{-p}(r^\gamma k_0)^p + c
%\]
%for the modular Poincar\'e inequality. Since $[ku]_{C^{0,\gamma}(B_r)} = k_0$, we arrive as before at the 
%necessary inequality $\alpha\ge (1-\gamma)(q-p)$, which is \aonen{\frac 2{1-\gamma}}. Thus also the assumption in 
%case (1) is sharp.
\end{example}

%%%%%%%%%%%%%%%%%%%%%%%%%%%%%%%%%%%%%%%%%%%%%%%%%%%%%%%%%%%%%%%%%%%%%
\subsection*{Quasiminimizers}\label{subsect:quasiminimzer}

In this subsection, we derive regularity results for
quasiminimizers with a priori information.
Let $\phi\in\Phiw(\Omega)$.
% of the following non-autonomous functional 
% \begin{equation}\label{functional1}
% \int_\Omega \phi(x,|Dv|)\,dx
% \end{equation}
%for some $L\ge 1$ and $\phi\in \Phiw(\Omega)$ satisfying an \aone{}-type condition.
%In this subsection we collect results on non-autonomous problems with Uhlenbeck structure.
%For $\phi\in\Phiw(\Omega)$, 
We say that $u\in W^{1,\phi}_{\loc}(\Omega)$ is a \textit{(local) quasi\-mini\-mizer} if there exists $Q\ge 1$ 
such that 
\begin{equation*}%\label{def:quasiminimizer}
\int_{\supp\,(u-v)} \phi(x,|Du|)\,dx \le Q \int_{\supp\,(u-v)} \phi(x,|Dv|)\,dx 
\end{equation*}
for every $v\in W^{1,\phi}_{\loc}(\Omega)$ with $\supp\,(u-v)\Subset\Omega$. 
Quasiminimizers of energy functionals with generalized Orlicz growth 
have been studied e.g.\ in \cite{BenHHK21, BenK20, HarHL21, HarHT17, HarHK18}.
If $\phi$ satisfies \adec{q}, then the quasiminimizer $u$ satisfies the \textit{Caccioppoli inequality} 
\begin{equation}\label{eq:caccio}
\int_{B_r} \phi(x,|Du|)\, dx \le c \int_{B_{2r}} \phi\bigg(x,\frac{|u-(u)_{B_{2r}}|}{r}\bigg)\, dx,
\end{equation}
for some $c=c(n,q,L,Q)\ge 1$ and every $B_{2r}\Subset \Omega$, see \cite[Lemma 4.6]{HarHL21}.

If $\phi$ satisfies \aonen{n}, then a bounded qusiminimizer satisfies a Harnack-type inequality 
and so is locally H\"older continuous \cite[Theorem 4.1]{HarHL21}. The main ingredients of the proof are the Caccioppoli estimate \eqref{eq:caccio} and the Sobolev--Poincar\'e inequality (Theorem~\ref{thm:poincare}). 
In Lemma~\ref{lem:poincare_sufficient}, we derived several sufficient conditions for the Sobolev--Poincar\'e inequality. Therefore, we obtain the H\"older continuity under these conditions from almost the same proof as \cite[Theorem 4.1]{HarHL21}. 
We start with local boundedness of quaisiminimizers. The proof is exactly the same as 
\cite[Proposition~5.5]{HarHL21} and is hence omitted. 

%Hence we only sketch the proof of the next theorem.

%\begin{theorem} \label{thm:quasiminimizer_holder} 
%Let $\phi\in \Phiw(\Omega)$ satisfy \azero{}, \ainc{p}, \adec{q} and \aonen{s} with $1<p\le q$. 
%Assume that $u\in W^{1,\phi}_{\loc}(\Omega)$ is a quasiminimizer and one of the following holds:
%\begin{enumerate}
%\item
%$s=n$ and $[u]_{BMO(\Omega)}\le b_0$;
%\item
%$s\in (n(1-\frac pq), n)$ and $u\in L^{s^*}(\Omega)$;
%\item
%$s > n(1-\frac pq)$ and $u \in W^{1,s}(\Omega)$.
%\end{enumerate}
%Then $u\in C^{0,\gamma}_{\loc}(\Omega)$ for some $\gamma\in(0,1)$ depending only on $n$, $p$, $q$, $L$,
%$L_\omega$ and $Q$, and in Case~(1) additionally on $b_0$.
%\end{theorem}
%
%\begin{remark}
%If $u\in L^{s^*}(\Omega)$, then there exists $r_0>0$ such that $\|u\|_{L^{s^*}(B_r)}\le 1$ for any $B_r\Subset\Omega$ with $0<r \le r_0$. Therefore, in the above theorem, since we consider a local regularity property, we can find the H\"older exponent $\gamma$ independent of $\|u\|_{L^{s^*}(\Omega)}$ or 
%$\|u\|_{W^{1,s}(\Omega)}$ in Cases (2) and (3). 
%\end{remark}
%
%\comment{Jihoon: The proof of the first theorem is rather simple and the second theorem is a corollary of \cite[Theorem 4.1]{HarHL21}.\\
%Peter: If $s>n$, then we do not obtain \aonen{n} from \aonen{s}, but the Holder continuity is trivial.
%}
%

\begin{lemma} \label{lem:bounded}
Let $\phi\in \Phiw(\Omega)$ satisfy \azero{}, \ainc{p}, \adec{q} and \aonen{s} with $1<p \le q$, and $s\in(1,n]$.
Assume that $u\in W^{1,\phi}_{\loc}(\Omega)$ is a quasiminimizer and 
$u\in BMO(\Omega)$ with $s=n$, or 
$u\in L^{s^*}(\Omega)$ with $s\in (n(1-\frac pq), n)$.
Then $u\in L^{\infty}_{\loc}(\Omega)$. 
\end{lemma}

\begin{theorem} \label{thm:holder} 
%Let $\phi\in \Phiw(\Omega)$ satisfy \azero{}, \ainc{p}, \adec{q} and \aonen{n} with $1<p\le q$. 
%Assume that $u\in W^{1,\phi}_{\loc}(\Omega)\cap L^\infty(\Omega)$ is a quasiminimizer.
Let $\phi\in \Phiw(\Omega)$ satisfy \azero{}, \ainc{p}, \adec{q} and \aonen{s} with $1<p \le q$, and $s\in(1,n]$.
Assume that $u\in W^{1,\phi}_{\loc}(\Omega)$ is a quasiminimizer and one of the following holds:
\begin{enumerate}
\item
$s=n$ and $u\in BMO(\Omega)$.
\item
$s\in (n(1-\frac pq), n)$ and $u\in L^{s^*}(\Omega)$.
%\item
%$s > n(1-\frac pq)$ and $u \in W^{1,s}(\Omega)$.
\end{enumerate}
For every $\Omega'\Subset\Omega$ there exists $\gamma\in(0,1)$ depending only on $n$, $p$, $q$, $L$, $L_\omega$, $Q$ and $\Omega'$ such that $u\in C^{0,\gamma}(\Omega')$. 
\end{theorem}
\begin{proof}
Assume first that $\phi$ satisfies \aonen{n} and $u\in W^{1,\phi}_{\loc}(\Omega)\cap L^\infty(\Omega)$.
Then local H\"older continuity follows directly from the Harnack inequality in \cite[Theorem 4.1]{HarHL21}.
We consider then assumptions (1) and (2) and note that Lemma~\ref{lem:bounded} implies that 
$u\in L^{\infty}_{\loc}(\Omega)$. 
Furthermore, \aonen{s} implies \aonen{n} when $s<n$. Therefore, we obtain the local H\"older continuity 
from the result for bounded solutions. 
%Specifically, for each $\Omega'\Subset\Omega''\Subset\Omega$, $u\in C^{0,\gamma}(\Omega')$ for some $\gamma\in(0,1)$ depending only on $n$, $p$, $q$, $L$, $L_\omega$, $Q$, $\Omega'$, $\Omega''$ and $\|u\|_{L^\infty(\Omega'')}$.
\end{proof}

We end the section with a higher integrability result for 
H\"older continuous quasiminimizers. We first observe that if $u\in C^{0,\gamma}(B_{2r})$ 
for some $\gamma\in(0,1)$ and $\phi$ satisfies \aonen{\frac{n}{1-\gamma}}, then by Jensen's 
inequality and the Caccioppoli inequality \eqref{eq:caccio}
\[ \begin{split}
\phi^-_{B_{2r}}\left(\fint_{B_r} |Du|\, dx\right) &\lesssim \fint_{B_r} \phi^{-}_{B_{2r}}(|Du|)\, dx 
\lesssim \fint_{B_{2r}} \phi\bigg(x,\frac{|u-(u)_{B_{2r}}|}{r}\bigg)\, dx\\
&\le \fint_{B_{2r}} \phi (x,[u]_{\gamma}(2r)^{\gamma-1} )\, dx 
\lesssim \big([u]_{\gamma}^{p} + [u]_{\gamma}^{q}\big) \fint_{B_{2r}} \phi (x,r^{\gamma-1})\, dx \\
& \lesssim \big([u]_{\gamma}^{p} + [u]_{\gamma}^{q}\big) (\phi^-_{B_{2r}} (r^{\gamma-1})+1)
%\lesssim \phi^-_{B_{2r}}(r^{\gamma-1})+1
% \approx \phi^-_{B_{2r}}(r^{\gamma-1}+1),
\end{split}\]
which implies 
\begin{equation}\label{estimateDu}
\fint_{B_r} |Du|\, dx \le  c\, \Big([u]_{\gamma}^{\frac pq} + [u]_{\gamma}^{\frac qp}\Big) (r^{\gamma-1}+1).
\end{equation}
Here $c$ depends on $n$, $p$, $q$, $L$, $L_\omega$ and $Q$, and is independent of $\gamma$. 

\begin{theorem}[Reverse H\"older inequality]\label{thm:reverseHolder} 
Let $\phi\in \Phiw(\Omega)$ satisfy \azero{}, \ainc{p}, \adec{q} and 
\aonen{\frac{n}{1-\gamma}} with $1<p\le q$ and $\gamma\in(0,1)$. 
If $u\in W^{1,\phi}_{\loc}(\Omega)\cap C^{0,\gamma}(B_{4r})$ is a 
quasiminimizer in $B_{4r}\Subset\Omega$, then 
$|Du|\in L^{\phi^{1+\sigma}}(B_r)$ for some $\sigma>0$ with the estimate 
\begin{equation*}%\label{revestimate}
\left(\fint_{B_r} \phi(x,|Du|)^{1+\sigma}\,dx\right)^{\frac{1}{1+\sigma}} 
\le
c\phi_{B_{2r}}^-\left(\fint_{B_{2r}} |Du|\,dx\right) +c.
\end{equation*}
The constants $\sigma$ and $c$ depend only on $n$, $p$, $q$, $L$, $L_\omega$, $Q$ and $[u]_{\gamma, B_{2r}}$.
\end{theorem}

\begin{proof}
By the Caccioppoli estimate \eqref{eq:caccio} and 
the Sobolev--Poincar\'e inequality (from Remark~\ref{rem:sobopoin} and Lemma~\ref{lem:poincare_sufficient}(1)), 
we obtain
\[
\fint_{B_{\rho}} \phi(x,|Du|)\, dx 
\le
c_1 \bigg(\fint_{B_{2\rho}} \phi(x,|Du|)^{\frac{1}{1+\theta_1}}\, dx \bigg)^{1+\theta_1}+1
\] 
for every ball $B_{2\rho}\subset B_{2r}$, where $\theta_1$ and $c_1$ depend 
on the parameters listed in the statement.
By Gehring's Lemma (e.g.\ \cite[Theorem 6.6]{Gi1}), there exists $\sigma>0$ depending on $\theta_1$ and $c_1$ such that 
\[
\bigg(\fint_{B_\rho} \phi(x,|Du|)^{1+\sigma}\, dx\bigg)^{\frac{1}{1+\sigma}} 
\lesssim
\fint_{B_{2\rho}} \phi(x,|Du|)\, dx +1
\] 
for every ball $B_{2\rho}\subset B_{2r}$. Moreover, using the technique from 
\cite[Lemma~4.7]{HasO22} with \aonen{\frac{n}{1-\gamma}} instead of \aone{} and with \eqref{estimateDu} in 
$B_\rho$, we obtain for every ball $B_{2\rho}\subset B_{2r}$ that
\[
\bigg(\fint_{B_\rho} \phi(x,|Du|)^{1+\sigma}\, dx\bigg)^{\frac{1}{1+\sigma}} 
 \lesssim \phi^+_{B_{2\rho}}\bigg(\fint_{B_{2\rho}} |Du|\, dx\bigg)+1 
 \lesssim \phi^-_{B_{2\rho}}\bigg(\fint_{B_{2\rho}} |Du|\, dx\bigg) +1. \qedhere
\]
\end{proof}

\section{Growth functions and autonomous problems}\label{sect:growthFunctions}

Let us precisely define our solutions and minimizers. 
Since we only consider local versions we will drop the word ``local'' later on, as indicated by the 
parentheses.
We say that $u\in W^{1,1}_{\loc}(\Omega)$ is a \textit{(local) weak solution to \eqref{mainPDE}} if 
$|Du|\,|A(\cdot,Du)|\in L^1_{\loc}(\Omega)$ and
\[
\int_{\Omega} A(x,Du)\cdot D\zeta \,dx = 0 
\]
for all $\zeta\in W^{1,1}(\Omega)$ with $\supp\zeta \Subset \Omega$ and $|D\zeta|\,|A(\cdot, D\zeta)| \in L^1(\Omega)$. We say that 
$u\in W^{1,1}_{\loc}(\Omega)$ 
is a \emph{(local) minimizer} if $F(\cdot,Du)\in L^1_{\loc}(\Omega)$ and
\[
\int_{\supp(u-v)} F(x,Du)\, dx \le \int_{\supp(u-v)} F(x,Dv)\, dx
\]
for every $v\in W^{1,1}_{\loc}(\Omega)$ with $\supp (u-v)\Subset\Omega$. 
Note that if \eqref{mainPDE} is an Euler--Lagrange equation, that is, if $A=D_\xi F$ for some a function $F$, 
then the weak solution to \eqref{mainPDE} is a minimizer of \eqref{mainfunctional}. 

We introduce fundamental assumptions on $A:\Omega\times\Rn\to\Rn$ or $F:\Omega\times \Rn\to[0,\infty)$ 
with respect to the gradient variable $\xi$  from \cite{HasO22,HasO22b}, so-called \textit{$(p,q)$-growth} and 
\textit{quasi-isotropy} conditions (parts (Aii) and (Aiii) of the definition, respectively).
Here, ``quasi-isotropic'' indicates that $A$ or $F$ can be estimated by a non-autonomous isotropic 
$\Phi$-function, the so-called \textit{growth function}.

\begin{definition}\label{def:AF}
We say that $A:\Omega\times\R^n\to \R^n$ or $F:\Omega\times\R^n\to [0,\infty)$ 
has \textit{quasi-isotropic $(p, q)$-growth} if conditions (Ai)--(Aiii) or (Fi)--(Fii) hold, respectively.
\begin{enumerate}
\item[(Ai)]
For every $x\in \Omega$, $A(x, 0)= 0$ and $A(x,\cdot)\in C^{1}(\R^n\setminus\{ 0\}; \R^n)$
and for every $\xi\in \R^n$, $A(\cdot,\xi )$ is measurable.
\item[(Aii)]
There exist $L\ge 1$ and $1<p\le q$ such that the radial function 
$t\mapsto | D_\xi A(x,te)|$ satisfies \azero{}, \ainc{p-2} and \adec{q-2} with the constant $L$,
for every $x\in \Omega$ and $e\in \partial B_1(0)$. 
\item[(Aiii)]
There exists $L\ge 1$ such that 
\[
| D_\xi A(x,\xi ')|\,|\tilde\xi|^2
\le
L\, D_\xi A(x,\xi )\tilde\xi \cdot \tilde\xi 
\]
for all $x\in\Omega$, $\xi ,\xi', \tilde\xi\in \Rn\setminus\{0\}$ with $|\xi |=|\xi '|$. 
\item[(Fi)] 
For every $x\in \Omega$, $F(x, 0)=|D_\xi F(x,0)|= 0$ and 
$F(x,\cdot)\in %C^{1}(\R^n)\cap 
C^{2}(\R^n\setminus\{ 0\})$
and for every $\xi\in \R^n$, $F(\cdot,\xi)$ is measurable.
\item[(Fii)]
The derivative $A:=D_\xi F$ satisfies conditions (Aii) and (Aiii).
\end{enumerate}
\end{definition}

The assumptions (Aii) and (Aiii) for $A:=D_\xi F$ impose the two crucial conditions on the 
Hessian matrix $D^2_\xi F$.
The former means that $|\xi|^2 D^2_{\xi}F(x,\xi)$ satisfies a $(p,q)$-growth condition which is a variant of $(p,q)$-growth of $F$. 
The latter is equivalent to the existence of $\tilde L\ge 1$ such that
\begin{equation}\label{puelliptic}
\frac{\sup\{\text{eigenvalue of }D^2_\xi F(x,\xi)\,:\, \xi\in \partial B_t(0)\}}
{\inf\{\text{eigenvalue of }D^2_\xi F(x,\xi)\,:\, \xi\in \partial B_t(0)\}}\le \tilde L \quad \text{for all }\ x\in\Omega \ \text{ and } \ t>0.
\end{equation}
Note that all examples in Table~\ref{table:examples} satisfy this condition, but not
\begin{equation}\label{guelliptic}
\frac{\sup\{\text{eigenvalue of }D^2_\xi F(x,\xi)\,:\,  x\in \Omega, \xi\in \partial B_t(0)\}}
{\inf\{\text{eigenvalue of }D^2_\xi F(x,\xi)\,:\, x\in \Omega, \xi\in \partial B_t(0)\}} \le \tilde L \quad \text{for all } \ t>0,
\end{equation}
which holds for $F(x,\xi)\approx a(x)\psi(|\xi|)$ with $0<\nu\le a \le L$. 
We remark \eqref{puelliptic} and \eqref{guelliptic} are called the \textit{pointwise} and \textit{global} uniform ellipticity condition. Here, ``uniform" is concerned with the variable $|\xi|$. For more discussion about this uniform ellipticity condition, we refer to \cite{DeFMin22}.

\begin{definition}\label{def:growthfunction}
Let $A:\Omega\times \Rn \to \Rn$ or $F:\Omega\times \R\to [0,\infty)$ have quasi-isotropic $(p, q)$-growth.
We say that $\phi\in \Phic(\Omega)$ is its \textit{growth function} if 
there exist $1<p_1\le q_1$ and $0<\nu\le \Lambda$ such that 
$\phi(x,\cdot)\in C^1([0,\infty))$ for every $x\in\Omega$ and 
$\phi'$ satisfies \azero{}, \inc{p_1-1} and \dec{q_1-1} as well as 
\begin{equation*}%\label{Agrowth}
|A(x,\xi)|+ |\xi|\,|D_\xi A(x,\xi)|\le \Lambda \phi'(x,|\xi|)
\quad\text{and}\quad
D_\xi A(x,\xi)\tilde \xi \cdot \tilde \xi \ge \nu \frac{\phi'(x,|\xi|)}{|\xi|}|\tilde \xi|^2
\end{equation*} 
for all $x\in \Omega$ and $\xi, \tilde \xi\in \Rn \setminus\{ 0\}$; 
in the case of $F$ we assume that the the inequalities hold for $A:=D_\xi F$. 
\end{definition}

In this paper we always use the growth function from the following proposition. Thus the 
additional parameters $p_1$, $q_1$, $\nu$ and $\Lambda$ only depend on the original 
parameters $p$, $q$ and $L$.

\begin{proposition}[Proposition~3.3, \cite{HasO22b}]\label{prop:growthfunction} 
Every $A:\Omega\times \Rn \to \Rn$ and $F:\Omega\times \R\to [0,\infty)$ with quasi-isotropic $(p, q)$-growth 
has a growth function $\phi\in \Phic(\Omega)$ with $p_1=p$, $q_1\ge q$, 
$\nu$ and $\Lambda$ depending only on $p$, $q$ and $L$.
\end{proposition}

By Proposition~\ref{prop0}, the growth function $\phi$ satisfies \azero{}, \inc{p}, \dec{q_1} as well as $\phi^*(x,\phi'(x,t))\le \phi'(x,t)t\approx \phi(x,t)$.
Furthemore, by \cite[Remark~3.4]{HasO22b} we have the strict monotonicity condition
\begin{equation}\label{monotonicity}
(A(x,\xi)-A(x,\tilde \xi)) \cdot (\xi-\tilde \xi) \gtrsim \frac{\phi'(x,|\xi|+|\tilde \xi|)}{|\xi|+|\tilde \xi|}|\xi - \tilde \xi|^2, \qquad x\in\Omega, \ \ \xi,\tilde \xi\in \Rn\setminus \{0\},
\end{equation}
as well as the equivalences 
\begin{equation}\label{phiEquiv}
%\phi(x,|\xi|) \lesssim A(x,\xi)\cdot \xi \le |\xi| |A(x,\xi)| \le |\xi|\phi'(x,|\xi|) \approx \phi(x,|\xi|), \qquad x\in\Omega,\ \ \xi\in \Rn,
\phi(x,|\xi|) \approx A(x,\xi)\cdot \xi \approx |\xi| |A(x,\xi)|
\quad\text{and}\quad 
F(x,\xi)\approx \phi(x,|\xi|), \qquad x\in\Omega,\ \ \xi\in \Rn,
\end{equation} 
where the implicit constants depend on only $p$, $q$ and $L$. 
%\marginpar{independent $n$? Should be.}
%if $\phi$ is from Proposition~\ref{prop:growthfunction}.
%Let $F:\Omega\times \Rn \to [0,\infty)$ satisfy Assumption~\ref{ass:AF} and 
%$\phi\in\Phic(\Omega)$ be its growth function. By \eqref{phiAequiv} with $A=D_\xi F$,
%\begin{equation}\label{phifequiv}
%F(x,\xi) = \int_0^1 D_\xi F(x,t\xi)\cdot \xi\,dx \approx \phi(x,|\xi|) \quad \text{for all $x\in\Omega$ and $\xi\in\Rn$,}
%\end{equation}

Let $A:\Omega\times\Rn\to\Rn$ have quasi-isotropic $(p, q)$-growth and growth function $\phi$ 
and let $u\in W^{1,1}_{\loc}(\Omega)$ be a weak solution to \eqref{mainPDE}. 
We showed in \cite[Section~4.1]{HasO22b} that $u$ is a quasiminimizer of the $\phi$-energy 
for some $Q=Q(p,q,L)\ge 1$; in the reference we assumed that smooth functions are dense in 
the Sobolev space $W^{1,\phi}$, which is reasonable if $\phi$ satisfies \aone{}. 
But this is not needed here due to our changed test function class in the definition of weak solution. 
Thus we can apply the regularity results for quasiminimizers from Section~\ref{sect:lower} 
to weak solutions.

We next show how the \aone{}-type condition of $A$ or $F$ transfers to the growth function $\phi$. 
Based on Part~(2), we can say that \VAn{\psi} and  \wVAn{\psi} are weaker in the minimization 
case than in the PDE case. This justifies studying minimizers separately. 

\begin{proposition} \label{prop:phiVA}
Let $A:\Omega\times \Rn \to \Rn$ or $F:\Omega\times \R\to [0,\infty)$ 
have quasi-isotropic $(p,q)$-growth and growth function $\phi$, 
and let $\psi:\Omega\times \Rn \to [0,\infty)$. 
\begin{enumerate}
\item 
$A$ or $F$ satisfies \aonen{\psi} if and only if $\phi$ satisfies \aonen{\psi}.
%In particular, if $A$ satisfies \aonen{\psi}, we have that for each $B_r$ with $r\in (0,1]$ and $t\ge 1$ with $\psi^-_{B_r}(t)\le |B_r|^{-1}$,
%\[
%\phi^+_{B_r}(t) \le c \phi^-_{B_r}(t) ;
%\]
%
\item
If $A:= D_\xi F$ satisfies \VAn{\psi}, then so does $F$, with the same $\omega$ up to a constant depending on $p$, $q$ and $L$.
\item 
If $A$ or $F$ satisfies \VAn{s}, $s>0$, and $\theta\in (0,1]$, then 
\[
|A(x,\xi)-A(y,\xi)| \le c\, \omega(r)^\theta\big(\phi'(y,|\xi|)+1\big)
\]
or
\[
|F(x,\xi)-F(y,\xi)| \le c\, \omega(r)^\theta\big(\phi(y,|\xi|)+1\big)
\]
for $B_r$ with $r\in (0,1]$, $x,y\in B_r\cap\Omega$ and $\xi\in\R^n$ with 
$3^n\omega(r)^{n(1-\theta)} |\xi|^s\in [0,|B_r|^{-1}]$. 
The constant $c>0$ depends only on $p$, $q$ and $L$. 
\end{enumerate}
\end{proposition}

\begin{proof}
We present only the proofs for $A$, as the ones for $F$ are similar but simpler. 
All implicit constants in the proof depend only on $p$, $q$ and $L$. 
Fix $x,y\in B_r\cap \Omega$ and $\xi\in\Rn$ with 
$\psi(y,\xi)\in [0,|B_r|^{-1}]$.

We first prove (1). 
Suppose that $A$ satisfies \aonen{\psi} and abbreviate $t:=|\xi|$. By \aonen{\psi} of $A$ 
and the equivalence \eqref{phiEquiv},
\[\begin{split}
\phi(x,t) & \approx t |A(x,\xi)| \le t |A(x,\xi)-A(y,\xi)| +t |A(y,\xi)| \\
&\lesssim t (|A(y,\xi)| + 1) + \phi(y,t)
\lesssim \phi(y,t) + t
\lesssim \phi(y,t)+1;
\end{split}\]
in the last inequality we used \azero{} and \inc{1} of $\phi$ when $t>1$. 
Thus $\phi$ satisfies \aonen{\psi}.
Conversely, suppose $\phi$ satisfies \aonen{\psi}. Then we see, using also \azero{} and \ainc{1}, that 
\[
|A(x,\xi)| \approx \frac{\phi(x,|\xi|)}{|\xi|} \lesssim \begin{cases} 
1 \quad \text{if }\ |\xi|\le 1, \\
\frac{\phi(y,|\xi|)+1}{|\xi|} \lesssim \frac{\phi(y,|\xi|)}{|\xi|} \approx |A(y,\xi)| \quad \text{if }\ |\xi| > 1.
\end{cases}
\]
Hence $A$ satisfies \aonen{\psi}.

We omit the proof of (2) which is essentially the same as \cite[Proposition~3.8]{HasO22b}.

To prove (3), we note by \VAn{s} and Proposition~\ref{prop:K} that
$
|A(x,\xi)-A(y,\xi)| 
\le 
\omega(r)^\theta (|A(y,\xi)|+1) 
\approx 
\omega(r)^\theta (\phi'(y,|\xi|) +1)$ when $|\xi|$ satisfies the condition. 
\end{proof}

When considering equations with nonlinearity $A$
it is natural to assume \VAn{\psi} for the function 
$A^{(-1)}(x,\xi):=|\xi| A(x,\xi)$. This is the route we took in \cite{HasO22b}. However, the next 
result shows that the conditions for $A$ and $A^{(-1)}$ are equivalent, up to an exponent which does not affect the conclusion of the main results in the next section. Thus we will in the rest of article use the assumptions directly for $A$. 

\begin{proposition}
Let $A:\Omega\times \Rn \to \Rn$ have quasi-isotropic $(p,q)$-growth.
Then $A$ satisfies \VAn{\psi} if and only if $A^{(-1)}$ does. 
If the original modulus continuity is $\omega$, then the new modulus continuity can be taken 
as $c\,\omega^{\frac{1}{p'}}$ for some $c=c(p,q,L,L_\omega)>0$.
\end{proposition}

\begin{proof} 
Let $x,y\in B_r\cap \Omega$ and  $\psi(y,\xi)\in (0,|B_r|^{-1}]$.
Assume $A$ satisfies \VAn{\psi}.
Since $|\xi| \le c |\xi| |A(y,\xi)|+1$,
\[\begin{split}
|A^{(-1)}(x,\xi)- A^{(-1)}(y,\xi)| &= |\xi | |A(x,\xi)- A(y,\xi)| \\
&\le \omega(r)  \left(|\xi| |A(y,\xi)|+|\xi|\right) \le c \omega(r)  \left(|A^{(-1)}(y,\xi)|+1\right)
\end{split}\] 
This gives \VAn{\psi} of $A^{(-1)}$, with modulus of continuity $c\omega$. 

We prove opposite implication through three cases. If $|\xi|\ge1$, then
\[
|\xi|\,|A(x,\xi)- A(y,\xi)| 
\le 
\omega(r)(|\xi| |A(y,\xi)|+1) 
\le \omega(r) |\xi| (|A(y,\xi)|+1). 
\]
Hence we suppose that $|\xi| < 1$. 
If further $|A(x,\xi)|,|A(y,\xi)|\le \omega(r)^\frac1{p'}$, then 
\[
|A(x,\xi)- A(y,\xi)| \le 2\omega(r)^\frac1{p'}.
\]
Otherwise, we use \azero{} and \ainc{p-1} to deduce 
$c|\xi|^{p-1} \ge \max\{|A(x,\xi)|,|A(y,\xi)|\}\ge \omega(r)^\frac1{p'}$. 
Thus $1\le c |\xi|\omega(r)^{-\frac1p}$ so that
\[
|\xi|\,|A(x,\xi)- A(y,\xi)| 
\le \omega(r) (|\xi| |A(y,\xi)|+1)
\le  c \omega(r)^{1-\frac1p}  |\xi| ( |A(y,\xi)|+1).
\]
Dividing both sides by $|\xi|$ gives the desired estimate. 
\end{proof}

\begin{remark}
In \cite[Proposition~3.8]{HasO22b} we assumed that $A^{(-1)}$ 
satisfies \wVA{} but in the proof we used the condition for $A$. 
This mistake can be corrected by means of the previous proposition.
\end{remark}

%%%%%%%%%%%%%%%%%%%%%%%%%%%%%%%%%%%%%%%%%%%%%%%%%%%%%%%%%%%%%%%%%%%%%%%%
%%%%%%%%%%%%%%%%%%%%%%%%%%%%%%%%%%%%%%%%%%%%%%%%%%%%%%%%%%%%%%%%%%%%%%%%
%%%%%%%%%%%%%%%%%%%%%%%%%%%%%%%%%%%%%%%%%%%%%%%%%%%%%%%%%%%%%%%%%%%%%%%%

%{sect:preliminaries}{sect:lower}{sect:growthFunctions}{sect:auxiliary}

%\section{Auxiliary results for autonomous problems}\label{sect:auxiliary}
Having defined the structure conditions, we first consider quasi-isotropic $(p,q)$-growth for 
an autonomous function $\bA:\Rn\to \Rn$ via its trivial extension $\bA(x,\xi):= \bA(\xi)$. It has a
growth function $\bphi\in \Phic\cap C^1([0,\infty))$, cf.~\cite{HasO22b}.
For such $\bA$ and $\bphi$, we present regularity results of weak solutions to 
\begin{equation}\label{eq:A0}
\tag{{\(\div \bA\)}}
\mathrm{div}\, \bA (D\bu)=0 \quad \text{in }\ B_r.
\end{equation}
Note that we use the bar-symbol to indicate the autonomous versions 
of $A$, $F$, $\phi$ and corresponding solutions or minimizers $u$. 
In the Uhlenbeck case $\bA(\xi)=\frac{\bphi'(|\xi|)}{\xi}\xi$, we proved the next result in 
\cite{HasO22} and in \cite{HasO22b} we sketched how to extend the proof to the quasi-isotropic case. 
Since the result is independent of the \aone{}-type assumptions, it applies directly also 
to this paper.

%The corresponding result for the Uhlenbeck case $\bA(\xi)=\bA(|\xi|)$ was considered in \cite[Lemma~4.12]{HasO22}.

\begin{lemma}[$C^{1,\alpha}$-regularity, Lemma~4.4, \cite{HasO22b}]\label{lem:holder}
Let $\bA:\R^n\to\R^n$ have quasi-isotropic $(p,q)$-growth and $\bphi\in \Phic$ be its growth function. 
If $\bu{}\in W^{1,\bphi}(B_r)$ is a weak solution to \eqref{eq:A0}, then 
$D\bu\in C^{0,\bar\alpha}_{\loc}(B_r,\Rn)$ for some $\bar\alpha\in (0,1)$ with the following estimates: 
\[
\sup_{B_{\rho/2}} |D\bu|\leq c\fint_{B_\rho} |D\bu|\, dx
\qquad\text{and}\qquad
\fint_{B_{\tau \rho}}\big|D\bu- (D\bu)_{B_{\tau \rho}} \big|\,dx 
\leq 
c \tau^{\bar\alpha}\fint_{B_\rho}|D\bu|\,dx
\]
for every $B_\rho\subset B_r$ and $\tau\in(0,1)$. 
Here $\bar\alpha$ and $c>0$ depend only on $n$, $p$, $q$ and $L$.
\end{lemma}
%\comment{Is there really a dependence on $L$ in the previous lemma? \azero{} should be irrelevant when 
%there is no x-dependence...

%The next result is a Calder\'on--Zygmund-type estimate in the generalized Orlicz space for non-zero boundary data. 
%Since $\theta$ is superlinear, this lemma allows us to transfer regularity from $u$ to
%$\bu$. 
%%It generalizes \cite[Lemma~4.15]{HasO22}. 
%
%\begin{lemma} \label{lem:CZ} (Calder\'on-Zygmund estimates, \cite[Lemma~4.5]{HasO22b})
% Let $\bA:\R^n\to\R^n$ satisfy Assumption~\ref{ass:AF} with $A(x,\xi)\equiv \bA(\xi)$ and constants $L\ge 1$ and $1<p\le q$, and $\bphi\in \Phic$ be its growth function. Suppose $\theta\in \Phiw(B_r)$ satisfies
%\azero{}, \ainc{p_\theta} and \adec{q_\theta} with constants $L_\theta\ge 1$ and $1<p_\theta\leq q_\theta$ and \aone{} with constant $L_K>0$, 
%and $u\in W^{1,\bphi}(B_r)$ satisfies $\int_{B_r}\theta(x,\bphi(|Du|))\,dx\le \kappa$ for some $\kappa>0$.
%If $\bu\in u+W^{1,\bphi}_0(B_r)$ is a weak solution to \eqref{eq:A0}, then 
%\[
%\|\bphi(|D\bu|)\|_{L^\theta(B_r)} \leq c\, \|\bphi(|Du|)\|_{L^\theta(B_r)}
%\]
%and
%\begin{equation*}%\label{meanCZestimate}
%\fint_{B_r}\theta(x,\bphi(|D\bu|))\,dx 
%\leq 
%c\big(\kappa^{\frac{q_1}{p}-1}+1\big)
%\bigg(\fint_{B_r}\theta(x,\bphi(|Du|))\,dx + 1\bigg)
%\end{equation*}
%for $c=c(n,p,q,L, p_\theta,q_\theta,L_\theta,L_K)>0$.
%\end{lemma}

%%%%%%%%%%%%%%%%%%%%%%%%%%%%%%%%%%%%%%%%%%%%%%%%%%%%%%%%%%%%%%%%%%%%%%%%%%%%

Next, we derive a harmonic approximation lemma for weak solutions and minimizers of 
autonomous problems. 
We start by recalling a Lipschitz truncation lemma, which is a formulation from 
\cite[Theorem 3.2]{DieSV12} and \cite[Theorem~5.2]{BarCM18} of the result in \cite{AceF88}, see also \cite{DieMS08,CruDie19}.

\begin{lemma}\label{lem:Liptrun} 
For $w\in W^{1,1}_0(B_r)$ and $\lambda>0$ there exist $w_\lambda\in W^{1,\infty}_0(B_r)$ and a 
zero-measure set $N$ such that
$\|Dw_\lambda\|_{L^\infty(B_r)}\leq c \lambda$ for some $c>0$ depending only on $n$ and 
\begin{equation*}
\{w_\lambda\neq w\} \, \subset\, \{ Mw>\lambda\}\cup N,
\end{equation*}
where $M$ is the Hardy--Littlewood maximal operator, 
$\displaystyle Mw(x):=\sup_{\rho>0} \fint_{B_\rho(x)}|w|\,dy$. 
\end{lemma} 

Part (1) of the next lemma, on almost solutions,
was considered in \cite[Lemma~1.1]{DieSV12} and \cite[Lemma~5.1]{BarCM18} in the Orlicz and 
double phase case, respectively. 
%proved in the double phase case $\bphi(t)=t^p+a(x) t^q$ in \cite[Lemma~5.1]{BarCM18}. 
We streamline and generalize their argument and include also almost minimizers in Part (2). 
The more precise estimates will 
allow us to omit the re-scaling step in the proofs of the main theorems.

\begin{lemma}[Harmonic approximation]\label{lem:har1} 
Let $\bA:\Rn\to\Rn$ or $\bF:\Rn\to [0,\infty)$ have quasi-isotropic $(p,q)$-growth, 
$\bphi\in \Phic$ be its growth function
and $\bu\in u+W^{1,\bphi}_0(B_r)$ be a weak solution to \eqref{eq:A0}, for $\bA:=D_\xi \bF$ 
in the case of $\bF$. Suppose that
there exist $b, \sigma >0$ and $\delta\in(0,1)$ for which 
\begin{equation*}%\label{lemhar1ass1}
\fint_{B_r} \bphi(|Du|)^{1+\sigma}\,dx \leq b^{1+\sigma}
\end{equation*}
and one of the following conditions holds for all $\eta\in W^{1,\infty}_0(B_r)$:
\begin{enumerate}
\item
$\displaystyle
\left| \fint_{B_r}\bA(Du) \cdot D\eta \,dx \right| 
\le \delta b \big\|\tfrac {D\eta}{\bphi^{-1}(b)}\big\|_{L^\infty(B_r)}^{\mu}$ for $\mu:=1$.
\item
$\displaystyle
\fint_{B_r} \bF(Du)\, dx 
\leq 
\fint_{B_r} \bF(Du+D\eta)\, dx 
+ \delta b\Big( \big\|\tfrac {D\eta}{\bphi^{-1}(b)}\big\|_{L^\infty(B_r)}+1\Big)^\mu$ for some $\mu>0$.
\end{enumerate}
Then there exist $c=c(n,p,q,L,\sigma)>0$ and $\bsigma:=\frac{\sigma p}{\mu+\sigma p}$
such that 
\[
\fint_{B_r}\frac{\bphi'(|Du|+|D\bu|)}{|Du|+|D\bu|}|Du-D\bu|^2\,dx 
\leq c \delta^{\bsigma} b .
\]
\end{lemma}

\begin{proof}
All implicit constants in this proof depend only on $n$, $p$, $q$, $L$ and $\sigma$. Let 
$w:= u-\bu \in W^{1,\bphi}_0(B_r)$ and let $\lambda\ge 1$ be a constant to be chosen. 
With these $w$ and $\lambda$ we consider $w_\lambda\in W^{1,\infty}_0(B_r)$ from 
Lemma~\ref{lem:Liptrun} so that $\|Dw_\lambda\|_{L^\infty(B_r)}\lesssim \lambda$. Denote
\[
\mathcal V:=\frac{\bphi'(|Du|+|D\bu|)}{|Du|+|D\bu|}|Du-D\bu|^2,\quad
W:=B_r\cap \{w\ne w_\lambda\}\quad\text{and}\quad W^c=B_r\cap \{w=w_\lambda\}.
\]

%Using H\"older's inequality, we find that
%\[
%\fint_{B_r} (\bphi(|Du|)+\bphi(|D\bu|)) \chi_W\,dx
%\le \bigg(\frac{|W|}{|B_r|}\bigg)^{\frac\sigma{1+\sigma}}
%\bigg(\fint_{B_r}\mathcal \bphi(|Du|+|D\bu|)^{1+\sigma} \,dx\bigg)^{\frac1{1+\sigma}}.
%\]
Since $\bphi$ is independent of $x$, we have
$\fint_{B_r} \bphi(|D\bu|)^{1+\sigma}\,dx \lesssim \fint_{B_r} \bphi(|Du|)^{1+\sigma}\,dx$ 
by a Calder\'on--Zygmund-type estimate in the Orlicz setting, see for instance 
%\cite[Lemma~B.1]{HasO22}
\cite[Lemma~4.5]{HasO22b} 
with $\theta(x,t)\equiv t^{1+\sigma}$.
This and the integrability assumption on $Du$ imply that 
\begin{equation*}%\label{lem:harmonic_pf1}
\bigg(\fint_{B_r} \bphi(|D\bu|)\,dx\bigg)^{1+\sigma} 
\le 
\fint_{B_r} \bphi(|D\bu|)^{1+\sigma}\,dx 
\lesssim 
\fint_{B_r} \bphi(|Du|)^{1+\sigma}\,dx \le b^{1+\sigma}.
\end{equation*}
To estimate $\frac{|W|}{|B_r|}$, we use $W\subset\{M(|Dw|)>\lambda\}\cup N$ 
from Lemma~\ref{lem:Liptrun} and the maximal estimate in 
$L^{\bphi^{1+\sigma}}$ \cite[Corollary~4.3.3]{HarH19}:
\begin{equation*}%\label{lem:harmonic_pf2}
\begin{split}
\frac{|W|}{|B_r|} 
 &\le \frac{|\{M (|Dw|)>\lambda\}\cap B_r|}{|B_r|} 
\le \fint_{B_r}\frac{\bphi(M(|Dw|))^{1+\sigma}}{\bphi(\lambda)^{1+\sigma}}\,dx\\
& \lesssim \frac{1}{\bphi(\lambda)^{1+\sigma}}\fint_{B_r} \bphi(|Dw|)^{1+\sigma}\,dx
\lesssim \frac{b^{1+\sigma}}{\bphi(\lambda)^{1+\sigma}},
\end{split}
\end{equation*}
where, in the last step we used $\bphi(|Dw|)\lesssim \bphi(|D\bu|)+\bphi(|Du|)$ and the 
earlier integrablilty estimates for $Du$ and $D\bu$. 
Using H\"older's inequality and these estimates, we find that
\begin{equation}\label{eq:West}
\begin{split}
&\fint_{B_r} [\bphi(|Du|)+\bphi(|D\bu|)+\bphi(\lambda)] \chi_W\,dx \\
&\qquad\le 
\bigg(\frac{|W|}{|B_r|}\bigg)^{\frac\sigma{1+\sigma}}
\bigg(\fint_{B_r} \big(\bphi(|Du|)+\bphi(|D\bu|)\big)^{1+\sigma} \,dx\bigg)^{\frac1{1+\sigma}}
+ \frac{|W|}{|B_r|} \bphi(\lambda) \\
&\qquad\lesssim 
\Big(\frac{b}{\bphi(\lambda)}\Big)^{\sigma}b+\frac{b^{1+\sigma}}{\bphi(\lambda)^{1+\sigma}}\bphi(\lambda)
\approx
\frac{b^{1+\sigma}}{\bphi(\lambda)^{\sigma}}.
\end{split}
\end{equation}

We first assume (1).
Using monotonicity \eqref{monotonicity} and that $\bu$ is a weak solution to \eqref{eq:A0}, we find that
\[
\begin{split}
\fint_{B_r}\mathcal V \chi_{W^c}\,dx &\lesssim \fint_{B_r}\big(\bA(Du)-\bA(D\bu)\big) \cdot (Du-D\bu)\chi_{W^c}\,dx\\
&= \fint_{B_r}\bA(Du) \cdot Dw_\lambda \,dx -\fint_{B_r}\big(\bA(Du)-\bA(D\bu)\big) \cdot Dw_\lambda \,\chi_{W}\,dx
\end{split}
\]
For the first term we use assumption (1) with $\eta=w_\lambda$, 
and for the second we use that growth functions satisfy $|\bA|\lesssim \bphi'$.
Continuing using Young's inequality, 
$\bphi^*(\bphi'(t))\lesssim \bphi(t)$ and $ \|Dw_\lambda\|_{L^\infty(B_r)}\lesssim \lambda$, 
we find that
\[
\begin{split}
\fint_{B_r}\mathcal V \chi_{W^c}\,dx 
&\lesssim 
\frac b{\bphi^{-1}(b)}\delta \|Dw_\lambda\|_{L^\infty(B_r)} +\fint_{B_r}\big(\bphi'(|Du|)+\bphi'(|D\bu|)\big) \|Dw_\lambda\|_{L^\infty(B_r)} \chi_{W}\,dx\\
&\lesssim 
\frac b{\bphi^{-1}(b)}\delta \lambda + \fint_{B_r}[\bphi(|Du|)+ \bphi(|D\bu|)+\bphi(\lambda)]\chi_W\,dx
\end{split}
\]
In $W$, we use $\mathcal V\lesssim \bphi(|Du|)+\bphi(|D\bu|)$ and obtain the same integral as on 
the right-hand side, above. 
With these estimates and \eqref{eq:West}, we obtain that 
\[
\fint_{B_r}\mathcal V \,dx
\le
\fint_{B_r}\mathcal V \chi_W\,dx
+
\fint_{B_r}\mathcal V \chi_{W^c}\,dx
\lesssim 
 \frac{b}{\bphi^{-1}(b)} \delta \lambda + 
\frac{b^{1+\sigma}}{\bphi(\lambda)^\sigma}.
\]
We choose $\lambda:=\bphi^{-1}(b\delta^{-\kappa})\le \delta^{-\kappa/p}\bphi^{-1}(b)$ 
with $\kappa:=\frac p{1+\sigma p}$, and find that
\[
\fint_{B_r}\mathcal V \,dx
\lesssim 
\Big(\delta\frac{\bphi^{-1}(b\delta^{-\kappa})}{\bphi^{-1}(b)} +\delta^{\kappa\sigma}\Big) b
\lesssim
\big(\delta^{1-\frac\kappa p}+\delta^{\kappa\sigma}\big) b
\approx
\delta^{\frac {\sigma p}{1+\sigma p}} b.
\]
This is the desired upper bound with 
$\bsigma := \frac{\sigma p}{1+\sigma p}$ and concludes the proof in Case (1).

Assume next that $\bA:=D_\xi \bF$ and (2) holds. 
We showed in \cite[Lemma~6.3]{HasO22b} that 
\begin{equation*}%\label{f0Talyor}
\frac{\bphi'(|\xi_1|+|\xi_2|)}{|\xi_1|+|\xi_2|}|\xi_1-\xi_2|^2
\lesssim
\bF(\xi_1)-\bF(\xi_2) -D_\xi \bF(\xi_2) \cdot (\xi_1-\xi_2).
\end{equation*}
Using this with $\bA=D_\xi \bF$, the weak form of \eqref{eq:A0}, $\bu=u+w_\lambda$ in $W^c$, 
$\bF\approx \bphi$, assumption (2), $|\bA|\lesssim \bphi'$ and Young's inequality with $\bphi^*(\bphi'(t))\lesssim \bphi(t)$, 
we have 
\[\begin{split}
\fint_{B_r}\mathcal V \chi_{W^c} \,dx 
&\lesssim 
\fint_{B_r} \left[\bF(Du) - \bF(D\bu)-\bA(D\bu)\cdot(Du-D\bu)\right]\chi_{W^c}\,dx\\
& = 
\fint_{B_r} [\bF(Du) -\bF(Du+Dw_\lambda)]\,dx\\
&\qquad -\fint_{B_r} \left[\bF(Du) -\bF(Du+Dw_\lambda)-\bA(D\bu)\cdot Dw_\lambda\right]\chi_W\,dx\\
& \lesssim 
b\delta\Big( \big\|\tfrac {Dw_\lambda}{\bphi^{-1}(b)}\big\|_{L^\infty(B_r)}+1\Big)^\mu
 +\fint_{B_r} \left[\bphi(|Du|) +\bphi(|D\bu|)+\bphi(\lambda)\right]\chi_W\,dx.
\end{split}\]
In $W$ we use the estimate $\mathcal V\lesssim \bphi(|Du|)+\bphi(|D\bu|)$ as before. 
With \eqref{eq:West} and the choice 
$\lambda:=\bphi^{-1}(b\delta^{-\kappa})\lesssim \delta^{-\kappa/p} \bphi^{-1}(b)$ for 
$\kappa:=\frac{p}{\mu+\sigma p}$, we obtain that 
\[
\fint_{B_r}\mathcal V \,dx 
\lesssim 
b\delta \big(\tfrac\lambda{\bphi^{-1}(b)}+1\big)^\mu + \frac{b^{1+\sigma}}{\bphi(\lambda)^{\sigma}}
\lesssim 
\big(\delta^{1-\frac{\kappa \mu}p}+\delta^{\sigma\kappa}\big)b
\approx 
\delta^{\frac{\sigma p}{\mu+\sigma p}}b.
\]
This is the desired upper bound with 
$\bar \sigma := \frac{\sigma p}{\mu+\sigma p}$ and concludes the proof in Case (2).
\end{proof}

%%%%%%%%%%%%%%%%%%%%%%%%%%%%%%%%%%%%%%%%%%%%%%%%%%%%%%%%%%%%%%%%%%%%%%%%
%%%%%%%%%%%%%%%%%%%%%%%%%%%%%%%%%%%%%%%%%%%%%%%%%%%%%%%%%%%%%%%%%%%%%%%%
%%%%%%%%%%%%%%%%%%%%%%%%%%%%%%%%%%%%%%%%%%%%%%%%%%%%%%%%%%%%%%%%%%%%%%%%

\section{Maximal regularity}\label{sect:regularity}

%%%%%%%%%%%%%%%%%%%%%%%%%%%%%%%%%%%%%%%%%%%%%%%%%%%%%%%%%%%%%%%%%%%%%%%%
%\subsection*{Approximating problems}\label{subsect:approx}

Let $B_r\Subset\Omega$ with $|B_r|\le 1$, $\gamma\in(0,1)$ and 
\[
t_K:= K |B_r|^{\frac{\gamma-1}{n}}\quad \text{for fixed }\ K\ge1.
\]
%Choose $x_K\in B_r$ such that $\phi(x_K,t_K)\le 2 \phi^-_{B_r}(t_K)$. 
For $\phi\in \Phic(\Omega)$ with $\phi'$ satisfying \azero{}, \inc{p-1} and \dec{q_1-1} 
with $1<p\le q_1$ we define
\begin{equation}\label{eq:bphi}
 \bphi(t):=\int_0^t \bphi'(s)\,ds \quad \text{with}\quad \bphi'(t):=
\begin{cases}
\phi'(x_0,t)&\text{if}\ \ 0\leq t\leq t_K,\\
\frac {\phi'(x_0,t_K)}{t_K^{p-1}}t^{p-1}&\text{if} \ \ t_K\leq t,
\end{cases}
\end{equation}
where $x_0$ is the center of $B_r=B_r(x_0)$.
The relationship between $\phi$ and $\bphi$ is analogous to that in \cite[Section~5]{HasO22} 
with $t_1=0$ and $t_2=t_K$. Due to our improved tools, we are able to simplify the 
argument concerning small values of $t$ by removing the case $t\in [0,t_1]$. 

\begin{proposition}\label{prop:bphi}
Let $t_K$, $\phi$ and $\bphi$ be as above. Suppose that 
$\phi^+_{B_r}(t)\le L_K \phi^-_{B_r}(t)$ 
for some $L_K\ge 0$ and all $t \in[1, t_K]$. 
\begin{enumerate}
\item 
$\bphi\in C^1([0,\infty))$ and $\bphi'$ satisfies \inc{p-1} and \dec{q_1-1}.
\item
$\bphi(t)= \phi(x_0,t)$ for all $t\le t_K$.
%\item
%$\bphi(t)\le {\color{blue}\frac{q_1}{p}L_K} \phi(x,t)$ for all {\color{blue}$t\ge 1$} and $x\in B_r$.
%\item $\bphi(t)\leq \frac{q_1}{p}\tilde L \phi(x,t)$ for all $(x,t)\in B_r\times [t_K,\infty)$.
\item 
$\bphi(t) \le \frac{q_1}p L_K\phi(x,t)+L$ for all $(x,t)\in B_r\times[0,\infty)$ and 
$W^{1,\phi}(B_r) \subset W^{1,\bphi}(B_r)$. 
\end{enumerate}
\end{proposition}
\begin{proof}
Parts (1) and (2) follow directly from the definition of $\bphi$ and the inclusion in (3) follows from 
the inequality.
If $t\le 1$, then the inequality in (3) follows from \azero{} and when $t\in[1,t_K]$, it follows from 
$\phi^+_{B_r}(t)\le L_K \phi^-_{B_r}(t)$  and (2). 
If $t > t_K$, we calculate
\[
\begin{split}
\bphi(t) 
&= 
\phi(x_0,t_K) + \int_{t_K}^t \frac {\phi'(x_0,t_K)}{t_K^{p-1}}s^{p-1}\, ds
= \phi(x_0,t_K) + \frac {t^p-t_K^p}{ t_K^p} \frac{t_K \phi'(x_0,t_K)}p\\
&
%\le \frac {t^p}{p t_K^p} \phi'(x_0,t_K)t_K
\le \frac{q_1}p\Big(\frac{t}{t_K}\Big)^p \phi(x_0,t_K)
\le \frac{q_1}pL_K\Big(\frac{t}{t_K}\Big)^p \phi^-_{B_r}(t_K)
\le \frac{q_1}pL_K \phi(x,t). \qedhere
\end{split}
\]
\end{proof}

We prove our main theorem on H\"older continuous weak solutions to \eqref{mainPDE}. 
The major novelties as follows: We use Proposition~\ref{prop:K} to deal with large values of 
the derivative; this allows us to handle the borderline double phase case $q-p=\frac\alpha{1-\gamma}$ 
with $a\in VC^{0,\alpha}$. Second, this is the first time that harmonic approximation 
from \cite{BarCM18} has been applied in the generalized Orlicz case. 
Third, optimizations in the formulations allow us to avoid many steps in previous proofs 
and present a much more streamlined argument. Fourth, we obtain $C^{1,\alpha}$-regularity with exponent $\alpha$ independent of the a priori information $[u]_\gamma$.

\begin{theorem}\label{thm:PDEholder}
Let $A:\Omega\times \Rn \to \Rn$ have quasi-isotropic $(p,q)$-growth and $u\in W^{1,1}_{\loc}(\Omega)\cap C^{0,\gamma}(\Omega)$ be a weak solution to 
\eqref{mainPDE} with $\gamma\in(0,1)$. 
\begin{itemize}
\item[(1)] If $A$ satisfies \VAn{\frac{n}{1-\gamma}},
then $u\in C^{0,\alpha}_{\loc}(\Omega)$ for every $\alpha\in(0,1)$.
\item[(2)] If $A$ satisfies \VAn{\frac{n}{1-\gamma}} with $\omega(r)\lesssim r^{\beta}$ for some $\beta>0$, then $u\in C^{1,\alpha}_{\loc}(\Omega)$ for some $\alpha=\alpha(n,p,q,L,\gamma,\beta)\in(0,1)$.
\end{itemize}
\end{theorem}
\begin{proof} 
We prove the result in three steps. 
All implicit constants in the following estimates depend only on $n$, $p$, $q$, $L$ and $[u]_\gamma$.

\textit{Step~1, setting and approximating equation.}
For $\omega$ from \VAn{\frac n{1-\gamma}} we 
fix $B_{4r}\Subset \Omega$ with 
\begin{equation}\label{rrestriction}
r \in (0, 1),\quad 
\omega(4r)\le 1
\quad\text{and}\quad 
\omega(r)^{1-\theta}\leq \tfrac13,
\end{equation} 
where $\theta\in (0,1)$ is given in Step~2, below. Since we always consider $B_r$ with $\omega(4r)\le 1$, 
we assume without loss of generality that $A$ satisfies \aonen{\frac n{1-\gamma}} with $L_\omega=1$.
We set 
\[
K:=3^{-(1-\gamma)} \omega(r)^{-(1-\theta)(1-\gamma)}
\quad\text{so that}\quad
t_K=3^{-(1-\gamma)} \omega(r)^{-(1-\theta)(1-\gamma)} |B_r|^{\frac{\gamma-1}{n}}.
\]
Let $\phi\in\Phic(\Omega)$ be the growth function of $A$ from 
Proposition~\ref{prop:growthfunction} so that $\phi(x,\cdot)\in C^1([0,\infty))$ and
$\phi'$ satisfies \azero, \inc{p-1} and \dec{q_1-1}.
By Proposition~\ref{prop:phiVA}(1)\&(3), $\phi$ satisfies \aonen{\frac n{1-\gamma}} with $L_\omega$ depending only on $p$, $q$ and $L$ and 
\[
|A(x,\xi)-A(y,\xi)|\le \omega(r)^\theta \big(\phi'(y,|\xi|)+1\big) 
\quad\text{when }\ |\xi| \le t_K. 
\]
By \azero{} of $\phi'$, $|A|\approx \phi'$ and $\omega(r)\le 1$ we conclude that 
\[ 
\phi(x,t) \approx t |A(x,te_1)| \le t (|A(y,te_1)|+\phi'(y,t)+1) \lesssim \phi(y,t)
\]
for every $t\in[1, t_K]$ and $x,y\in B_r$. Thus we can apply Proposition~\ref{prop:bphi} 
with $L_K>0$ depending only on $p$, $q$ and $L$. 
 
Let $\bphi\in\Phic$ be from \eqref{eq:bphi}. By \cite[Lemma~5.2]{HasO22b} with 
$t_1=0$ and $t_2=t_K$ there exists an autonomous nonlinearity 
$\bA\in C(\Rn,\Rn)\cap C^{1}(\Rn\setminus\{ 0\},\Rn)$ having quasi-isotropic $(p,q_1)$-growth such that 
$\bphi$ is its growth function and
\[
\bA(\xi)=A(x_0,\xi) \quad \text{whenever }|\xi| \le \tfrac{1}{2}t_K. 
\]
The exact form of $\bA$ is given in \cite[(5.2)]{HasO22b}.

By Theorem~\ref{thm:reverseHolder}, $\phi(\cdot,|Du|)\in L^{1+\sigma}(B_r)$ for some $\sigma>0$ 
depending only on $n$, $p$, $q$, $L$, and by \eqref{estimateDu} in $B_{2r}$, 
\begin{equation}\label{Duestimate}
 \fint_{B_{2r}}|Du|\,dx
\lesssim 
r^{\gamma-1}
\lesssim t_K.
%|B_r|^{\frac{\gamma-1}{n}}.
%\Big([u]_{\gamma, 4r}^{\frac pq} + [u]_{\gamma, 4r}^{\frac qp}\Big) r^{\gamma-1}.
\end{equation}
Let $\bu\in W^{1,\bphi}(B_r)$ be the weak solution to \eqref{eq:A0} in $B_r$ with 
boundary value $u$.
Then $\bu$ is a quasiminimizer of the $\bphi$-energy and so the results of Section~\ref{sect:lower} can be applied. 
By Jensen's inequality, the minimization property and Proposition~\ref{prop:bphi}(3), 
\[
\bphi\left(\fint_{B_r}|D\bu|\,dx\right) 
\le \fint_{B_r}\bphi(|D\bu|)\,dx \lesssim \fint_{B_r}\bphi(|Du|)\,dx 
\lesssim \fint_{B_r}\phi(x,|Du|)\,dx+1.
\]
Then we use the reverse H\"older inequality (Theorem~\ref{thm:reverseHolder}), 
\azero{}, Proposition~\ref{prop:bphi}(2) with \eqref{Duestimate} and \dec{q_1} of $\phi$ 
to conclude that
\[
\bphi\left(\fint_{B_r}|D\bu|\,dx\right) 
\lesssim \phi^-_{B_{2r}}\left(\fint_{B_{2r}}|Du|\,dx+1\right) 
\approx \bphi\left(\fint_{B_{2r}}|Du|\,dx+1\right).
\]
It follows that 
\begin{equation}\label{eq:DvDuholder}
\fint_{B_r}|D\bu|\,dx \lesssim \fint_{B_{2r}}|Du|\,dx + 1 \lesssim r^{\gamma-1}.
\end{equation}

Next we set
\begin{equation*}%\label{defJ}
J:= \bphi^{-1}\left(\bigg(\fint_{B_{r}}\phi(x,|Du|)^{1+\sigma}\,dx\bigg)^{\frac{1}{1+\sigma}}+1\right).
\end{equation*}
By the reverse H\"older inequality (Theorem~\ref{thm:reverseHolder}) and 
the earlier estimate \eqref{eq:DvDuholder}, 
%Proposition~\ref{prop:bphi}(2) with \eqref{Duestimate},
\begin{equation}\label{JBr}
J \approx \bphi^{-1} \left(\phi^{-}_{B_{r}}\left(\fint_{B_{2r}}|Du|\,dx \right)+1\right) \approx
\fint_{B_{2r}}|Du|\,dx+1 \lesssim r^{\gamma-1},
\end{equation}
and, since $J\gtrsim 1$, 
\begin{equation}\label{phi0Japprox}
\begin{split}
\phi^+_{B_{r}}(J) \approx \phi^-_{B_{r}}(J) \approx \bphi(J) 
\approx \phi^-_{B_{r}}\left(\fint_{B_{2r}}|Du|\,dx+1\right) \lesssim \phi^-_{B_r}(r^{\gamma-1}).
\end{split} 
\end{equation}

\textit{Step~2, harmonic approximation.}
We prove that $u$ is an almost weak solution to \eqref{eq:A0} in the sense that 
\begin{equation*}%\label{almostharmonic_u}
\left|\fint_{B_r} \bA(Du) \cdot D\eta \,dx\right| \le c \bomega(r) \tfrac{\bphi(J)}{J} \|D\eta\|_{\infty}
\quad\text{with}\quad \bomega(r):= \omega(r)^{\frac{(1-\gamma)\sigma(p-1)}{(1-\gamma)\sigma(p-1)+1}}
\end{equation*}
for some $c\ge 1$ depending on $n$, $p$, $q$, $L$ and $[u]_{\gamma,4r}$ and all 
$ \eta \in W^{1,\infty}_0(B_r)$.

It suffices to consider $\eta$ with $\|D\eta\|_\infty\leq 1$ by scaling. Since $u$ is a weak solution to \eqref{mainPDE}, 
\[
\begin{split}
\left|\fint_{B_r} \bA(Du) \cdot D\eta \,dx\right| &\le \left|\fint_{B_r} A(x,Du) \cdot D\eta \,dx\right| + \fint_{B_r} |A(x,Du)-\bA(Du)|\,dx \\
&=\fint_{E_1\cup E_2}|A(x,Du)-\bA(Du)|\,dx,
\end{split}
\]
where
$\displaystyle
E_1:=\left\{x\in B_r: |Du| \leq \tfrac{1}{2} t_K\right\} 
$ and $\displaystyle
E_2:= \left\{x\in B_r: |Du|> \tfrac{1}{2} t_K\right\}$.

We first consider $E_1$ so that $\bA=A(x_0,\cdot)$. By 
Proposition~\ref{prop:phiVA}(3) with $s:=\frac{n}{1-\gamma}$,
\[\begin{split}
\fint_{B_r}|A(x,Du)-\bA(Du)|\chi_{E_1}\,dx 
&=\fint_{B_r}|A(x,Du)-A(x_0,Du)|\,\chi_{E_1}\,dx \\
& \lesssim \omega(r)^\theta\bigg(\fint_{B_r} \phi'(x,|Du|)\,dx +1\bigg).
\end{split}\]
We abbreviate $\hat\phi:=\phi^+_{B_{r}}$ and estimate the integral on the right-hand side by 
$\phi'(x,t)\approx \phi(x,t)/t$, Jensen's inequality for 
$\hat\phi^*$, $\hat\phi^*( \phi(x,t)/t) \le \phi(x,t)$,
H\"older's inequality and estimate \eqref{phi0Japprox} for $J$:
\[
\fint_{B_r} \phi'(x,|Du|)\,dx
\lesssim (\hat\phi^*)^{-1}\left(\fint_{B_r} \phi(x,|Du|)\,dx\right) 
\lesssim(\hat\phi^*)^{-1}\left(\bphi(J) \right)\lesssim \frac{\bphi(J)}{J}. 
\]
Since $1\lesssim \tfrac{\bphi(J)}{J}$, the whole integral over $E_1$ can be bounded 
by $\omega(r)^\theta \tfrac{\bphi(J)}{J}$. 

In $E_2$ we estimate
\[
\bphi(\tfrac12 t_K)^\sigma\fint_{B_r} \phi(x,|Du|)\chi_{E_2} \,dx
\le
\fint_{B_r} \phi(x,|Du|)^{1+\sigma} \,dx
\lesssim
\bphi(J)^{1+\sigma}.
\]
Since $\frac{\bphi(J)}{\bphi(\frac12 t_K)} \lesssim 
\frac{\bphi(r^{\gamma-1})}{\bphi( Kr^{\gamma-1})} \le K^{-p}$
by \eqref{JBr} and \inc{p}, we obtain
\begin{equation}\label{eq:E2estimate}
\fint_{B_r}\phi(x,|Du|) \chi_{E_2}\,dx
\lesssim
 K^{-\sigma p}\bphi(J).
\end{equation}
By Proposition~\ref{prop:bphi}(3) and \azero{}, $\bphi(t)\lesssim \phi(x,t)$ when $t\ge 1$.
Using this, the $\hat\phi^*$-Jensen inequality as in $E_1$, the previous estimate and \ainc{1/p'} of $(\hat\phi^*)^{-1}$, 
we find that 
\[\begin{split}
\fint_{B_r}|A(x,Du)-\bA(Du)|\,\chi_{E_2}\,dx
& \lesssim \fint_{B_r}\tfrac{\phi(x,|Du|)+ \bphi(|Du|)}{|Du|}\chi_{E_2}\,dx
\lesssim \fint_{B_r}\tfrac{\phi(x,|Du|)}{|Du|}\chi_{E_2}\,dx \\
%&\lesssim \fint_{B_r}((\phi^*)^{-1}(x,\cdot)\circ \phi^*(x,\cdot))(\phi'(x,|Du|)) \chi_{E_3}\,dx \\
&\lesssim (\hat\phi^*)^{-1} \left(\fint_{B_r}\phi(x,|Du|) \chi_{E_2}\,dx\right) \\
&\lesssim 
K^{-\sigma (p-1)} 
(\hat\phi^*)^{-1} (\bphi(J))
\approx 
K^{-\sigma(p-1)} \frac{\bphi(J)}{J},
\end{split}
\]
where we used $(\hat\phi^*)^{-1}(\hat\phi(t))\approx \hat\phi(t)/t$ and \eqref{phi0Japprox} in the last step.
The desired estimate follows when we combine the estimates in $E_1$ and $E_2$, 
recall that $K\ge 3^{-1} \omega(r)^{-(1-\theta)(1-\gamma)} $ and choose 
$\theta:=\frac{(1-\gamma)\sigma(p-1)}{(1-\gamma)\sigma(p-1)+1}$.

\textit{Step~3, conclusion.} 
Applying Lemma~\ref{lem:har1}(1) with $b:=\bphi(J)$ to the inequality from Step~2, we obtain that
\[\begin{split}
\fint_{B_r} \frac{\bphi'(|D \bu|+|Du|)}{|D\bu|+|Du|} |Du-D\bu|^2 \,dx & 
\lesssim 
\bomega(r)^{\bsigma} \bphi(J) 
\approx 
\bomega(r)^{\bsigma} \bphi\left(\fint_{B_{2r}}|Du|\,dx+1\right),
\end{split}\]
where $\bsigma=\frac{\sigma p}{1+\sigma p}$.
By well-known techniques (see, e.g., \cite[Corollary~6.3]{HasO22} for details) 
this implies an $L^1$-estimate for the 
difference of the gradients of $u$ and $\bu$ from Step~1:
\begin{equation}\label{L1comparison}
\fint_{B_r} |Du-D\bu|\,dx 
\le 
c \,\bomega(r)^{\frac{\bsigma}{2q_1}} \left(\fint_{B_{2r}}|Du|\,dx +1 \right), 
\end{equation}
with $c\ge 1$ depending on 
$n$, $p$, $q$, $L$ and $[u]_{\gamma, 4r}$ and $\bomega$ from Step 2. Note that we can make $\bomega(r)^{\frac{\bsigma}{2q_1}}$ as small as we want by choosing $r$ small. Therefore, this inequality and the Lipschitz regularity of the $\bA$-solutions $\bu$ (the first estimate in Lemma~\ref{lem:holder}) with \eqref{eq:DvDuholder} imply local $C^{0,\alpha}$-regularity for every $\alpha\in(0,1)$ by known methods, see, e.g., \cite[Theorem~7.2]{HasO22}.
%and the derivative by established methods, 

We next assume that $\omega(r)\lesssim r^{\beta}$ for some $\beta>0$. Fix $\Omega'\Subset\Omega$. 
From Part (1) we obtain $u\in C^{0, \gamma'}(\Omega')$ for any $\gamma'\in (\gamma,1)$, 
and consider $B_{4r}\subset\Omega'$ with $r$ satisfying \eqref{rrestriction} and 
$[u]_{\gamma,4r,\Omega'}\le 1$. Then we obtain \eqref{L1comparison} with 
$\bomega(r)^{\frac{\bsigma}{2q_1}}\lesssim r^{\beta_0}$ for 
$\beta_0$ depending only $n$, $p$, $q$, $L$, $\gamma$ and $\beta$. 
This inequality and the H\"older regularity of the gradient of the $\bA$-solution $\bu$ 
(the second estimate in Lemma~\ref{lem:holder}) with \eqref{eq:DvDuholder} imply 
$u\in C^{1,\alpha}_{\loc}(\Omega')$ for any $\alpha\in(0, \min\{\beta_0,1-\gamma'\})$ 
by known methods, see, e.g., \cite[Theorem~7.4]{HasO22}. 
Since $\beta_0$ and $1-\gamma'$ are independent of the arbitrary set $\Omega'\Subset\Omega$, this implies 
that $u\in C^{1,\alpha}_{\loc}(\Omega)$ for some $\alpha\in(0,1)$ depending only on
 $n$, $p$, $q$, $L$, $\gamma$ and $\beta$.
\end{proof}

%%%%%%%%%%%%%%%%%%%%%%%%%%%%%%%%%%%%%%%%%%%%%%%%%%%%%%%%%%%%%%%%%%%%%%%%%%%%%%%%%%%%%
%\subsection*{H\"older continuous minimizers}\label{sec:functional}

Next we prove maximal regularity for H\"older continuous minimizers of \eqref{mainfunctional}. 
This is our second main result.

\begin{theorem}\label{thm:functionalholder}
Let $F:\Omega\times \Rn \to \R$ have quasi-isotropic $(p,q)$-growth and $u\in W^{1,1}_{\loc}(\Omega)\cap C^{0,\gamma}(\Omega)$ be a minimizer of \eqref{mainfunctional} with $\gamma\in(0,1)$. 
\begin{itemize}
\item[(1)] If $F$ satisfies \VAn{\frac{n}{1-\gamma}}, 
then $u\in C^{0,\alpha}_{\loc}(\Omega)$ for every $\alpha\in(0,1)$.
\item[(2)] If $F$ satisfies \VAn{\frac{n}{1-\gamma}} with $\omega(r)\lesssim r^{\beta}$ for some $\beta>0$, then $u\in C^{1,\alpha}_{\loc}(\Omega)$ for some $\alpha=\alpha(n,p,q,L,\gamma,\beta)\in(0,1)$.
\end{itemize}
\end{theorem}
\begin{proof}
The methodology is similar to Theorem~\ref{thm:PDEholder} except for the 
application of harmonic approximation. Hence we will take advantage many parts of that proof.

\textit{Step~1, setting and approximating functional.}
We use the same choice of $r$ and $K$ as in Theorem~\ref{thm:PDEholder} and define $J$ in the same way.
Let $\phi\in\Phi_c(\Omega)$ be the growth function of $F$ from Proposition~\ref{prop:growthfunction};
it satisfies the same properties as in Theorem~\ref{thm:PDEholder}.
% 
%so that $\phi(x,\cdot)\in C^1([0,\infty))$ and $\phi'$ satisfies \azero, \inc{p-1} and \dec{q_1-1}. Note that by Proposition~\ref{prop:phiVA}(1), $\phi$ satisfies \aonen{\frac{n}{1-\gamma}} and the assumption of Proposition~\ref{prop:bphi}, where $L_\omega$ and $L_K$ depend only on $p$, $q$ and $L$, and 
%\begin{equation}\label{improvedVA1F}
%|F(x,\xi)-F(y,\xi)|\leq \omega(r)^\theta \big(|F(y,\xi)|+1\big) 
%\quad\text{when }\ |\xi| \le t_K.
%\end{equation}
%with $t_K=K|B_r|^{\frac{\gamma-1}{n}}= 3^{-(1-\gamma)}\omega(r)^{-(1-\theta)(1-\gamma)}|B_r|^{\frac{\gamma-1}{n}}$.
In \cite[Lemma~5.3]{HasO22b} we constructed an autonomous function $\bF :\Rn\to [0,\infty)$ 
such that $\bphi\in\Phic$ from \eqref{eq:bphi} is its growth function and
\[
\bF(\xi)=F(x_0,\xi) \quad \text{whenever }  |\xi| \le \tfrac12 t_K. 
\] 
By Theorem~\ref{thm:reverseHolder}, $\phi(\cdot,|Du|)\in L^{1+\sigma}(B_r)$ for some $\sigma=\sigma(n,p,q,L,[u]_{\gamma,4r})\in(0,1)$. 
Let $\bu\in W^{1,\bphi}(B_r)$ be the minimizer of 
\begin{equation}\label{functional:f0bdy}
\min_{\bu\in u+W^{1,\bphi}_0(B_r)} \int_{B_r} \bF(D\bu)\,dx ,
\end{equation}
or, equivalently, the weak solution to \eqref{eq:A0} in $B_r$ with $\bA:=D_\xi\bF$ and 
boundary value given by $u$.

\textit{Step~2, harmonic approximation.} 
We prove that $u$ is an almost minimizer of \eqref{functional:f0bdy} in the sense
that there exists $c=c(n,p,q,L,[u]_{\gamma,4r})\ge 1$ such that 
\[
\fint_{B_{r}} \bF(Du)\,dx \le \fint_{B_{r}} \bF(Du+D\eta)\,dx 
+c\,\bomega(r) 
\bigg(\frac{\|D\eta\|_\infty}{J}+1\bigg)^{(1+\sigma)q_1} \bphi(J)
\] 
for every $\eta\in W^{1,\infty}_0(B_r)$, where 
$\bomega(r):= \omega(r)^{\frac{(1-\gamma)\sigma p}{(1-\gamma)\sigma p+1}}$.
%$\bomega(r):= \omega(r)^{\frac{(1-\gamma)\sigma(p-1)}{(1-\gamma)\sigma(p-1)+1}}$.

Let $E_1$ and $E_2$ be the sets from Step~2 of the proof of Theorem~\ref{thm:PDEholder}.
By $\bF\approx \bphi$ \eqref{phiEquiv}, the definition of $\bF$, Propositions~\ref{prop:bphi}(3) and \ref{prop:phiVA}(3), and the definition of $J$,
\[\begin{split}
\fint_{B_{r}} \bF(Du)\,dx
&\le 
\int_{B_r} [F(x_0,Du) \chi_{E_1} +c \phi(x,|Du|)\chi_{E_2}]\,dx\\
&\le 
\fint_{B_r} F(x,Du)\,dx +c \omega(r)^\theta \fint_{B_r} [\phi(x,|Du|)+1]\,dx
+ cK^{-\sigma p} \bphi(J) \\
&\le \fint_{B_{r}} F(x,D u)\,dx+c\left( \omega(r)^{\theta} + K^{-\sigma p} \right) \bphi(J),
\end{split}\]
where the estimate for the term in $E_2$ is from \eqref{eq:E2estimate}.

Next we obtain a similar estimate for $v:=u+\eta \in u + W^{1,\infty}_0(B_r)$. 
We define sets $E_i'$ like $E_i$ but with $|Du|$ 
replaced by $|Dv|$. 
Since $u$ is an $F$-minimizer and $F\approx \phi$ \eqref{phiEquiv}, 
\[
\fint_{B_{r}} F(x,Du)\,dx 
\le \fint_{B_{r}} F(x,Dv)\,dx
\le
\fint_{B_r} F(x,Dv)\chi_{E_1'}\,dx + c \fint_{B_r}  \phi(x,|Dv|)\chi_{E_2'} \,dx.
\]
In $E_1'$ we use Proposition~\ref{prop:phiVA}(3) with $\bF=F(x_0,\cdot)$:
\[
\fint_{B_r} F(x,Dv)\chi_{E_1'}\,dx 
\le
\fint_{B_{r}} \big[\bF(Dv)+c\omega(r)^\theta \big(\phi(x,|Dv|)+1\big)\big]\,dx.
\]
For the second term on the right-hand side, we use $v=u+\eta$ and  
$\phi(x,|Dv|)\lesssim \phi(x,|Du|)+\phi(x,|D\eta|)$. Thus
\[
\fint_{B_{r}} \phi(x,|Dv|)\,dx
\lesssim
\bphi(J)+
\fint_{B_{r}} \phi(x,|D\eta|)\,dx.
\]
We use \dec{q_1} of $\phi$ along with \eqref{phi0Japprox} to handle the integral with $D\eta$:
\[
\fint_{B_r} \phi(x,|D\eta|) \,dx
\le
\bigg(\fint_{B_r} \phi(x,|D\eta|)^{1+\sigma} \,dx\bigg)^\frac1{1+\sigma}
\lesssim
\bigg(\frac{\|D\eta\|_\infty}{J}+1\bigg)^{q_1}\bphi(J).
\]

In $E_2'$ we estimate
\[\begin{split}
\bphi(\tfrac12 t_K)^\sigma\fint_{B_r} \phi(x,|Dv|)\chi_{E_2'} \,dx
&\le
\fint_{B_r} \phi(x,|Dv|)^{1+\sigma} \,dx
\lesssim
\bphi(J)^{1+\sigma} + \fint_{B_r} \phi(x,|D\eta|)^{1+\sigma} \,dx.
\end{split}\]
With the estimate for $D\eta$ from the previous paragraph, this and 
$\frac{\bphi(J)}{\bphi(\frac12 t_K)} \lesssim K^{-p}$ give
\[
\fint_{B_r} \phi(x,|Dv|)\chi_{E_2'} \,dx
\lesssim
K^{-\sigma p} 
\bigg(\frac{\|D\eta\|_\infty}{J}+1\bigg)^{(1+\sigma)q_1} \bphi(J).
\]
Collecting the estimates from this and the previous paragraph, we arrive at 
\[
\fint_{B_{r}} F(x,Du)\,dx 
\le
\fint_{B_{r}} \bF(Dv)\,dx
+
c\,\left(\omega(r)^{\theta} + K^{-\sigma p}\right)
\bigg(\frac{\|D\eta\|_\infty}{J}+1\bigg)^{(1+\sigma)q_1} \bphi(J).
\]
Combining this with the previous paragraphs, the estimate $K\ge 3^{-1}\omega(r)^{-(1-\theta)(1-\gamma)}$ and the choice of $\theta:=\frac{(1-\gamma)\sigma p}{(1-\gamma)\sigma p+1}$, 
%from Theorem~\ref{thm:PDEholder}, 
we complete this step.

\textit{Step~3, conclusion.} 
We apply Lemma~\ref{lem:har1}(2) with $b:=\bphi(J)$ and $\mu:=(1+\sigma)q_1$ to 
conclude that 
\[
\fint_{B_r}\frac{\bphi'(|Du|+|D\bu|)}{|Du|+|D\bu|}|Du-D\bu|^2 \,dx \lesssim 
\bomega(r)^{\bsigma} \left(\bphi\left(\fint_{B_{2r}}|Du|\,dx\right)+1\right)
\]
where $\bsigma :=\frac{\sigma p}{(1+\sigma)q_1 +\sigma p}$ and $\bomega$ is from Step 2. 
The estimate 
\[
\fint_{B_r}|Du-D\bu| \,dx \leq c\bomega(r)^{\frac\bsigma{2q_1}} \left(\fint_{B_{2r}}|Du|\,dx+1\right)
\]
follows from this as in \cite[Corollary~6.3]{HasO22}.
We complete the proof as in Step 3 of the proof of Theorem~\ref{thm:PDEholder}. 
Hence we omit details. 
\end{proof}

We apply the previous theorem to the double phase energies from Example~\ref{ex:doublephase}.
%The proof demonstrates how we can use the vanishing H\"older continuity of the solution 
%instead of the vanishing H\"older continuity of the coefficient as in the example. 

\begin{corollary}
Let $1<p\le q$, $a\ge 0$ and $\gamma\in(0,1)$. Assume that either
\begin{itemize}
\item
$F(x,\xi):=|\xi|^p + a(x) |\xi|^q$, 
where $a\in VC^{0,\alpha}(\Omega)$, $\alpha\in (0,1)$ and $q-p \le \frac{1}{1-\gamma}\alpha$; or
\item
$F(x,\xi):= |\xi|^p + a(x)^\alpha |\xi|^q$, 
where $a\in C^{0,1}(\Omega)$, $\alpha>1$ and $q-p < \frac{1}{1-\gamma}\alpha$.
\end{itemize}
If $u\in W^{1,1}_{\loc}(\Omega)\cap C^{0,\gamma}(\Omega)$ is a minimizer of \eqref{mainfunctional}, 
then $u\in C^{1,\beta}_{\loc}(\Omega)$ for some $\beta$ depending only on $n$, $p$, $q$, $\alpha$ and $\gamma$.
\end{corollary}

\begin{proof}
Consider first the case with $\alpha<1$. 
By Example~\ref{ex:doublephase}, $F$ satisfies 
\VAn{\frac{n}{1-\gamma}}. Therefore $u\in C^{0,\beta}_{\loc}(\Omega)$ 
for every $\beta\in(0,1)$ by the previous theorem. Fix $\gamma_1\in(\gamma,1)$. 
Then $F$ satisfies 
\VAn{\frac{n}{1-\gamma_1}} with $\omega(r)\lesssim r^{\alpha-(q-p)(1-\gamma_1)}$ 
and $\alpha-(q-p)(1-\gamma_1)> \alpha-(q-p)(1-\gamma)\ge 0$. 
Since $u\in C^{0,\gamma_1}_\loc(\Omega)$, the previous theorem yields that 
$u\in C^{1,\beta}_{\loc}(\Omega)$ for some 
$\beta\in(0,1)$ depending on $n$, $p$, $q$, $L$, $\alpha$ and $\gamma$. 
In the case $\alpha>1$ we can directly apply the theorem since the 
strict inequality $q-p < \frac{1}{1-\gamma}\alpha$ implies that 
the condition \VAn{\frac{n}{1-\gamma}} holds with a $\omega$ of power-type (cf.\ Example~\ref{ex:doublephase}).
\end{proof}

%\subsection*{Bounded weak solutions and minimizers}
%
%The regularity results for bounded weak solutions and minimizers follow directly from Theorems \ref{thm:PDEholder} and \ref{thm:functionalholder}, and Theorem~\ref{lem:holder}.
%
%
%
%
%\begin{proof}[Proof of Corollary~\ref{cor:functionalbounded}]
%The corollary follows from Theorem~\ref{thm:functionalholder} by the arguments from the proof of 
%Corollary~\ref{cor:PDEbounded}.
%\end{proof}

%%%%%%%%%%%%%%%%%%%%%%%%%%%%%%%%%%%%%%%%%%%%%%%%%%%%%%%%%%%%%%%%%%%%%%%%%%
%\subsection*{BMO or \texorpdfstring{$L^{s^*}$}{Ls*} weak solutions and minimizers}

Finally, combining Theorems~\ref{thm:holder}, \ref{thm:PDEholder} and \ref{thm:functionalholder}, we obtain regularity results for BMO and $L^{s^*}$ weak solutions and minimizers. 
Note that the case $L^{s^*}$ with $s\le (n(1-\frac pq)$ remains open, due to a lack of
a Sobolev--Poincar\'e inequality.

\begin{corollary}\label{Cor:BMO}
Let $A:\Omega\times \Rn \to \Rn$ or $F:\Omega\times \Rn \to \R$ have quasi-isotropic $(p,q)$-growth 
and satisfy \wVAn{s}.
Suppose that $u\in W^{1,1}_{\loc}(\Omega)$ is a weak solution to \eqref{mainPDE} or a minimizer of \eqref{mainfunctional} and that one of the following holds:
\begin{enumerate}
\item
$s=n$ and $u\in BMO(\Omega)$,
\item
$s\in (n(1-\frac pq), n)$ and $u\in L^{s^*}(\Omega)$.
%\item
%$s>n(1-\frac pq)$ and $u\in W^{1,s}(\Omega)$.
\end{enumerate}
Then $u\in C^{0,\alpha}_{\loc}(\Omega)$ for every $\alpha\in(0,1)$. 

Furthermore, if \VAn{s'} holds with $\omega(r)\lesssim r^{\beta}$ 
for some $s'>s$ and $\beta>0$, then $u\in C^{1,\alpha}_{\loc}(\Omega)$ for some $\alpha\in(0,1)$ 
depending on $n$, $p$, $q$, $L$, $s$, $s'$ and $\beta$. 
% and, in Case (1), $b_0$. 
%Here, $\omega_{\epsilon}$ is the function $\omega$ from \VAn{s^{1+\epsilon}}.
\end{corollary}
\begin{proof}
Note that \wVAn{s} implies \aonen{s} and fix $\Omega'\Subset\Omega$. Since $u$ is a quasiminimizer 
of the isotropic problem, Theorem~\ref{thm:holder} implies that  
$u\in C^{0,\gamma}(\Omega')$ for some $\gamma\in(0,1)$. The condition 
\wVAn{s} with $s<\frac n{1-\gamma}$ implies \VAn{\frac n{1-\gamma}}. 
By Theorem~\ref{thm:PDEholder} or \ref{thm:functionalholder}, we obtain $u\in C^{0,\alpha}_{\loc}(\Omega')$ 
so that $u\in C^{0,\alpha}_{\loc}(\Omega)$ for every $\alpha\in(0,1)$.

Next, we suppose that $\omega(r)\lesssim r^{\beta}$ in \VAn{s'} and fix $\Omega'\Subset\Omega$. By the previous part, 
$u\in C^{0,\gamma}(\Omega')$ for any $\gamma\in(0,1)$. Furthermore, \VAn{\frac{n}{1-\gamma}} holds with 
$\omega(r)\lesssim r^{\beta}$ when $\gamma$ is chosen so large that $\frac{n}{1-\gamma}\ge s'$. Therefore, by Theorem~\ref{thm:PDEholder} or \ref{thm:functionalholder}, we obtain $u\in C^{1,\alpha}_{\loc}(\Omega')$ so that $u\in C^{1,\alpha}_{\loc}(\Omega)$ for some $\alpha\in(0,1)$ depending only on $n$, $p$, $q$, $L$, $\gamma$ and $\beta$. 
\end{proof}

\begin{remark}
For $C^{1,\alpha}$-regularity, the obtained exponent $\alpha$ is independent of the a priori information on weak solutions or minimizers. 
\end{remark}

%%%%%%%%%%%%%%%%%%%%%%%%%%%%%%%%%%%%%%%%%%%%%%%%%%%%%%%%%%%%%%%%%%%%%%%%%%%%%%%%%%%%%%%%%%%%%%%%%%%%%%%%%%%%%%%%%%%%%%%%%%%%%%%%%%%%%%

%
%\begin{theorem} \label{thm:holderbounded} 
%Let $\phi\in \Phiw(\Omega)$ satisfy \azero{}, \ainc{p}, \adec{q} and \aonen{n} with $1<p < q$. If $u\in W^{1,\phi}_{\loc}(\Omega)$ be a quasiminimizer of \eqref{functional1} with \eqref{fphi} and $u\in BMO(\Omega)$, then $u\in L^\infty_{\loc}(\Omega)$ Hence for some $\gamma\in(0,1)$ depending on $n$, $p$, $q$, $L$, $L_K$ and $\|u\|_\infty$. 
%\marginpar{Jihoon: Is it essential that $\gamma$ depends on $\|u\|_\infty$?}
%\end{theorem}
%
%\begin{corollary}
%In Corollaries \ref{cor:PDEbounded} and \ref{cor:functionalbounded}, we can replace the bounded assumption on $u$ by the weaker assumption that $u\in BMO(\Omega)$.
%\end{corollary}

%%%%%%%%%%%%%%%%%%%%%%%%%%%%%%%%%%%%%%%%%%%%%%%%%%%%%%%%%%%%%%%%%%%%%%%%
%%%%%%%%%%%%%%%%%%%%%%%%%%%%%%%%%%%%%%%%%%%%%%%%%%%%%%%%%%%%%%%%%%%%%%%%
%%%%%%%%%%%%%%%%%%%%%%%%%%%%%%%%%%%%%%%%%%%%%%%%%%%%%%%%%%%%%%%%%%%%%%%%

\section*{Acknowledgment and data statement}

%We thank the referee for comments. 
Peter H\"ast\"o was supported in part by the Jenny and Antti Wihuri Foundation and 
Jihoon Ok was supported by the National Research Foundation of Korea by the Korean Government (NRF-2022R1C1C1004523). We thank the referee for comments.

Data sharing is not applicable to this article as obviously no datasets were generated or analyzed during the current study.

%%%%%%%%%%%%%%%%%%%%%%%%%%%%%%%%%%%%%%%%%%%%%%%%%%%%%%%%%%%%%%%%%%%%%%%%
%%%%%%%%%%%%%%%%%%%%%%%%%%%%%%%%%%%%%%%%%%%%%%%%%%%%%%%%%%%%%%%%%%%%%%%%
%%%%%%%%%%%%%%%%%%%%%%%%%%%%%%%%%%%%%%%%%%%%%%%%%%%%%%%%%%%%%%%%%%%%%%%%

\bibliographystyle{amsplain}

\end{document}